\documentclass[krantz1]{krantz}
\usepackage{fixltx2e,fix-cm}
\usepackage{amssymb,amsmath,amsthm}%ashley
\usepackage{graphicx}
\usepackage{makeidx}
\usepackage{units}
\usepackage{enumerate}
\usepackage[arrow,curve,matrix,tips,2cell]{xy}
  \SelectTips{eu}{10} \UseTips
  \UseAllTwocells

\usepackage[bbgreekl]{mathbbol}
\usepackage{amsfonts}
\DeclareSymbolFontAlphabet{\mathbb}{AMSb}
\DeclareSymbolFontAlphabet{\mathbbl}{bbold}

\theoremstyle{plain}
\newtheorem{theorem}{Theorem}[section]
\newtheorem{lemma}[theorem]{Lemma}
\newtheorem{proposition}[theorem]{Proposition}
\newtheorem{corollary}[theorem]{Corollary}
\newtheorem{addendum}[theorem]{Addendum}

\theoremstyle{definition}
\newtheorem{definition}[theorem]{Definition}
\newtheorem{example}[theorem]{Example}
\newtheorem{remark}[theorem]{Remark}

{\catcode`@=11\global\let\c@equation=\c@theorem}

\newcommand{\id}{\operatorname{id}}
\newcommand{\pr}{\operatorname{pr}}
\newcommand{\tr}{\operatorname{tr}}
\newcommand{\map}{\operatorname{map}}
\newcommand{\Ext}{\operatorname{Ext}}
\newcommand{\colim}{\operatorname{colim}}
\newcommand{\op}{\operatorname{op}}

\newcommand{\THH}{\operatorname{THH}}
\newcommand{\TC}{\operatorname{TC}}
\newcommand{\TP}{\operatorname{TP}}

\newcommand{\HH}{\operatorname{HH}}

\newcommand{\Hom}{\operatorname{Hom}}
\newcommand{\Fil}{\operatorname{Fil}}
\newcommand{\gr}{\operatorname{gr}}

\begin{document}

\renewcommand{\leftmark}{Lars Hesselholt and Thomas Nikolaus}

\chapter{Topological cyclic homology}

\vspace{-12mm}
{\large \textsl{Lars Hesselholt and Thomas Nikolaus}}
\vspace{5mm}

{\small
\begin{quote}
\textsc{Abstract:} This survey of topological cyclic homology
is a chapter in the Handbook on Homotopy Theory. We give a brief
introduction to topological cyclic homology and the cyclotomic trace
map following Nikolaus-Scholze, followed by a proof of
B\"okstedt periodicity that closely resembles B\"okstedt's original
unpublished proof. We explain the extension of B\"{o}kstedt
periodicity by Bhatt-Morrow-Scholze from perfect fields to
perfectoid rings and use this to give a purely $p$-adic proof of Bott
periodicity. Finally, we evaluate the cofiber of the assembly map in
$p$-adic topological cyclic homology for the cyclic group of order $p$ and a
perfectoid ring of coefficients.
\end{quote}
}

\vspace{5mm}

\noindent Topological cyclic homology is a manifestation of Waldhausen's vision
that the cyclic theory of Connes and Tsygan should be developed with
the initial ring $\mathbb{S}$ of higher algebra as base. In his
philosophy, such a theory should be meaningful integrally as opposed
to rationally. B\"{o}kstedt realized this vision for Hochschild
homology~\cite{bokstedt}, and he made the fundamental calculation
that\index{B\"{o}kstedt periodicity}
$$\THH_*(\mathbb{F}_p) = \HH_*(\mathbb{F}_p/\hskip.5pt\mathbb{S}) =
\mathbb{F}_p[x]$$
is a polynomial algebra on a generator $x$ in degree
two~\cite{bokstedt1}. By comparison,
$$\HH_*(\mathbb{F}_p/\hskip.5pt \mathbb{Z}) = \mathbb{F}_p \langle x
\rangle$$
is the  divided power algebra,\footnote{\,The divided
power algebra\index{divided power algebra}
$\mathbb{F}_p \langle x \rangle$ has generators 
$x^{[i]}$ with $i \in \mathbb{N}$ subject to the relations that
$x^{[i]} \cdot x^{[j]} = \binom{i+j}{i} \cdot x^{[i+j]}$ for all $i,j
\in \mathbb{N}$ and $x^{[0]} = 1$. So $x^i = (x^{[1]})^i = i!
\cdot x^{[i]}$.} so B\"{o}kstedt's periodicity
theorem indeed shows that by replacing the base $\mathbb{Z}$
by the base $\mathbb{S}$, denominators disappear. In fact, the
base-change map
$\HH_*(\mathbb{F}_p/\hskip.5pt\mathbb{S}) \to
\HH_*(\mathbb{F}_p/\hskip.5pt\mathbb{Z})$
can be identified with the edge homomorphism of a spectral sequence
$$E_{i,j}^2 = \HH_i(\mathbb{F}_p/\pi_*(\mathbb{S}))_j \Rightarrow
\HH_{i+j}(\mathbb{F}_p/\hskip.5mm\mathbb{S}),$$
so apparently the stable homotopy groups of spheres have exactly the
right size to eliminate the denominators in the divided power algebra.

The appropriate definition of cyclic homology relative to $\mathbb{S}$
was given by
B\"{o}kstedt--Hsiang--Madsen~\cite{bokstedthsiangmadsen}. It involves
a new ingredient not present in the Connes--Tsygan cyclic theory: a
Frobenius map. The nature of this Frobenius map is now much better
understood by the work of 
Nikolaus--Scholze~\cite{nikolausscholze}. As in the Connes--Tsygan
theory, the circle group $\mathbb{T}$ acts on topological Hochschild
homology, and  negative topological cyclic homology and periodic
topological cyclic homology are defined to be the homotopy fixed
points and the Tate construction, respectively, of this action:
\index{topological cyclic homology!negative}
\index{negative topological cyclic homology}
\index{topological cyclic homology!periodic}
\index{periodic topological cyclic homology}
$$\TC^{-}(A) = \THH(A)^{h\mathbb{T}} \hskip9mm \text{and} \hskip9mm
\TP(A) = \THH(A)^{t\mathbb{T}}.$$
There is always a canonical map from the homotopy fixed points to the
Tate construction, but, after $p$-completion, the $p$th Frobenius
gives rise to another such map and topological cyclic homology is the
equalizer\index{topological cyclic homology}
$$\begin{xy}
(0,0)*+{ \TC(A) }="1";
(25,0)*+{ \TC^{-}(A) }="2";
(64,0)*+{ \TP(A)^{\wedge} = \prod_p \TP(A)_p^{\wedge}. }="3";
{ \ar "2";"1";};
{ \ar@<.7ex>^-{(\varphi_p)} "3";"2";};
{ \ar@<-.7ex>_-{\operatorname{can}} "3";"2";};
\end{xy}$$
Here ``$(-)^{\wedge}$'' and ``$(-)_p^{\wedge}$'' indicates profinite
and $p$-adic completion.

Topological cyclic homology receives a map from algebraic
$K$-theory, called the cyclotomic trace map.
\index{cyclotomic trace map} 
Roughly speaking, this map records traces of powers of matrices and
may be viewed as a denominator-free version of the Chern character. 
There are two theorems that concern the behavior of this map
applied to cubical diagrams of connective $\mathbb{E}_1$-algebras
in spectra. If $A$ is such an $n$-cube, then the theorems give
conditions for the $(n+1)$-cube
$$\xymatrix{
{ K(A) } \ar[r] &
{ \TC(A) } \cr
}$$
to be cartesian. For $n = 1$, the Dundas--Goodwillie--McCarthy
theorem~\cite{dundasgoodwilliemccarthy} 
\index{Dundas--Goodwillie--McCarthy theorem}
states that this is so provided $\pi_0(A_0) \to \pi_0(A_1)$ is a
surjection with nilpotent kernel.  And for $n = 2$, the Land--Tamme
theorem~\cite{landtamme}, 
\index{Land--Tamme theorem}
which strengthens theorems of Corti\~{n}as~\cite{cortinas} and
Geisser--Hesselholt~\cite{gh4}, states that this is so provided $A$ is 
cartesian and $\pi_0(A_1 \otimes_{A_0} A_2) \to \pi_0(A_{12})$ an
isomorphism. For $n \geq 3$, it is an open question to find conditions
on $A$ that make $K(A) \to \TC(A)$ cartesian. It is also not clear
that the conditions for $n = 1$ and $n = 2$ are optimal. Indeed, a
theorem of Clausen--Mathew--Morrow~\cite{clausenmathewmorrow} 
\index{Clausen--Mathew--Morrow theorem}
states that for every commutative ring $R$ and ideal $I \subset R$
with $(R,I)$ henselian, the square
$$\begin{xy}
(0,7)*+{ K(R) }="11";
(25,7)*+{ \TC(R) }="12";
(0,-7)*+{ K(R/I) }="21";
(25,-7)*+{ \TC(R/I) }="22";
{ \ar "12";"11";};
{ \ar "21";"11";};
{ \ar "22";"12";};
{ \ar "22";"21";};
\end{xy}$$
becomes cartesian after profinite completion. So in this case, the
conclusion of the Dundas--Goodwillie--McCarthy theorem holds under
a much weaker assumption on the $1$-cube $R \to R/I$. The
Clausen--Mathew--Morrow theorem may be seen as a $p$-adic analogue of
Gabber rigidity~\cite{gabber}. Indeed, for prime numbers $\ell$ that
are invertible in $R/I$, the left-hand terms in the square above
vanish after $\ell$-adic completion, so the Clausen--Mathew--Morrow
recovers and extends the Gabber rigidity theorem.

Absolute comparison theorems between $K$-theory and
topological cyclic homology begin with the calculation that for
$R$ a perfect commutative $\mathbb{F}_p$-algebra, the cyclotomic trace
map induces an equivalence
$$\xymatrix{
{ K(R)_p^{\wedge} } \ar[r] &
{ \tau_{\geq 0}\TC(R)_p^{\wedge}. } \cr
}$$
The Clausen--Mathew--Morrow theorem then implies that the same is true
for every commutative ring $R$ such that $(R,pR)$ is henselian and
such that the Frobenius $\varphi \colon R/p \to R/p$ is
surjective. Indeed, in this case, $(R/p)^{\operatorname{red}}$ is
perfect and $(R/p,\operatorname{nil}(R/p))$ is henselian;
see~\cite[Corollary~6.9]{clausenmathewmorrow}. In particular, this is
true for all semiperfectoid rings $R$.\footnote{\,A commutative
$\mathbb{Z}_p$-algebra $R$ is perfectoid
(resp.~semiperfectoid)\index{perfectoid ring}\index{semiperfectoid
  ring}, for example,
if there exists a non-zero-divisor $\pi \in R$ with $p \in \pi^pR$
such that the $\pi$-adic topology on $R$ is complete and separated
and such that the Frobenius $\varphi \colon R/\pi \to R/\pi^p$ a
bijection (resp.~surjection).}

The starting point for the calculation of topological cyclic homology
and its variants is the B\"{o}kstedt periodicity theorem, which we
mentioned above. Since B\"{o}kstedt's paper~\cite{bokstedt1} has never
appeared, we take the opportunity to give his proof in
Section~\ref{sec:bokstedtperiodicity} below. The full scope of this
theorem was realized only recently by
Bhatt--Morrow--Scholze~\cite{bhattmorrowscholze1}, who proved that
B\"{o}kstedt periodicity holds for every perfectoid
ring. \index{B\"{o}kstedt periodicity} More precisely, their result,
which we explain in Section~\ref{sec:perfectoidrings} below states
that if $R$ is perfectoid, then\footnote{\,We write
$\pi_*(X,\mathbb{Z}_p)$ for the homotopy groups of 
the\index{$p$-adic homotopy groups} $p$-completion of a spectrum $X$,
and we write $\THH_*(R,\mathbb{Z}_p)$ instead of
$\pi_*(\THH(R),\mathbb{Z}_p)$. For $p$-completion, see
Bousfield~\cite{bousfield}.}
$$\THH_*(R,\mathbb{Z}_p) = R[x]$$
on a polynomial generator of degree $2$. Therefore, as is
familiar from complex orientable cohomology theories, the Tate
spectral sequence
$$E_{i,j}^2 = \hat{H}^{-i}(B\mathbb{T},\THH_j(R,\mathbb{Z}_p))
\Rightarrow \TP_{i+j}(R,\mathbb{Z}_p)$$
collapses, since all non-zero elements are concentrated in even
total degree. However, since the $\mathbb{T}$-action on
$\THH(R,\mathbb{Z}_p)$ is non-trivial, the ring homomorphism given by
the edge homomorphism of the spectral sequence,
$$\xymatrix{
{ \TP_0(R,\mathbb{Z}_p) } \ar[r]^-{\theta} &
{ \THH_0(R,\mathbb{Z}_p),} \cr
}$$
does not admit a section, and therefore, we cannot identify the domain
with a power series algebra over $R$. Instead, this ring homomorphism is
canonically identified with the universal $p$-complete  
pro-infinitesimal thickening\index{$p$-complete pro-infinitesimal thickening}
$$\xymatrix{
{ A = A_{\operatorname{inf}}(R) } \ar[r]^-{\theta} &
{ R } \cr
}$$
introduced by Fontaine~\cite{fontaine}, which we recall in
Section~\ref{sec:perfectoidrings} below. In addition,
Bhatt--Scholze~\cite{bhattscholze} show that, rather than a formal
group over $R$, there is a canonical $p$-typical $\lambda$-ring
structure 
$$\xymatrix{
{ A } \ar[r]^-{\lambda} &
{ W(A), } \cr
}$$
the associated Adams operation $\varphi \colon A \to A$ of which is
the composition of the inverse of the canonical map
$\operatorname{can} \colon \TC_0^{-}(R,\mathbb{Z}_p) \to
\TP_0(R,\mathbb{Z}_p)$ and the Frobenius map $\varphi \colon
\TC_0^{-}(R,\mathbb{Z}_p) \to \TP_0(R,\mathbb{Z}_p)$. The kernel
$I \subset A$ of the edge homomorphism $\theta \colon A \to R$ is a
principal ideal, and Bhatt--Scholze show that the pair
$((A,\lambda),I)$ is a prism\index{prism} in the sense that
$$p \in I + \varphi(I)A,$$
and that this prism is perfect in the sense that
$\varphi \colon A \to A$ is an isomorphism. We remark that if
$\xi \in I$ is a generator, then, equivalently, the prism condition
means that the  intersection of the divisors ``$\xi = 0$'' and
``$\varphi(\xi) = 0$'' is contained in the special fiber
``$p = 0$,'' whence the name. We thank Riccardo Pengo for the
following figure, which illustrates $\operatorname{Spec}(A)$.

\vspace{4mm}
\includegraphics[width=.92\textwidth]{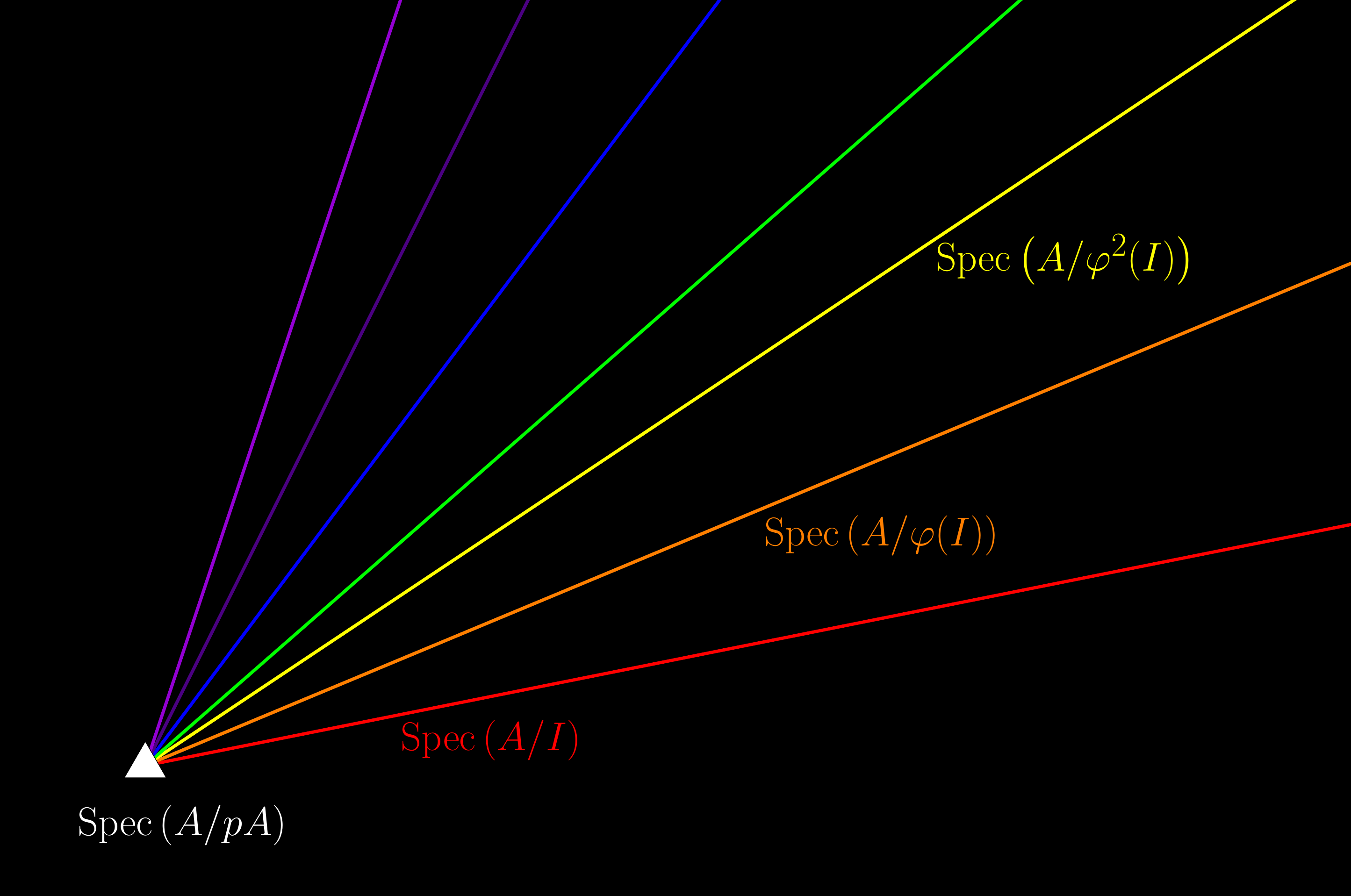}
\vspace{3mm}

\noindent Bhatt--Morrow--Scholze further show that, given the choice
of $\xi$, one can choose $u \in \TC_2^{-}(R,\mathbb{Z}_p)$,
$v \in \TC_{-2}^{-}(R,\mathbb{Z}_p)$, and
$\sigma \in \TP_2(R,\mathbb{Z}_p)$ in such a way that
$$\begin{aligned}
\TC_*^{-}(R,\mathbb{Z}_p) & = A[u,v]/(uv - \xi), \cr
\TP_*(R,\mathbb{Z}_p) & = A[\sigma^{\pm1}], \cr
\end{aligned}$$
and $\varphi(u) = \alpha \cdot \sigma$,
$\varphi(v) = \alpha^{-1}\varphi(\xi) \cdot \xi^{-1}$,
$\operatorname{can}(u) = \xi \cdot \sigma$, and
$\operatorname{can}(v) = \sigma^{-1}$ with $\alpha \in A$ a unit. In
these formulas, the unit $\alpha$ can be eliminated, if one is willing to
replace the generator $\xi \in I$ by the generator
$\varphi^{-1}(\alpha) \cdot \xi \in I$. We use these results for
$R = \mathcal{O}_C$, where $C$ is a complete algebraically closed
$p$-adic field, to give a purely $p$-adic proof of Bott
periodicity. In particular, B\"{o}kstedt periodicity implies
Bott periodicity, but not vice versa.

The Nikolaus--Scholze approach to topological cyclic homology is also
very useful for calculations. To wit, Speirs has much
simplified the calculation of the topological cyclic homology of
truncated polynomial algebras over a perfect
$\mathbb{F}_p$-algebra~\cite{speirs}, and we have evaluated the
topological cyclic homology of planar cuspical curves over a perfect
$\mathbb{F}_p$-algebra~\cite{hn}. Here, we illustrate this approach in
Section~\ref{sec:grouprings}, where we identify the cofiber of the
assembly map
$$\xymatrix{
{ \TC(R) \otimes BC_{p+} } \ar[r] &
{ \TC(R[C_p]) } \cr
}$$
for $R$ perfectoid in terms of an analogue of the affine deformation
to the normal cone along $I \subset A = A_{\operatorname{inf}}(R)$ with
$p$ as the parameter.\goodbreak

Finally, we mention that
Bhatt--Morrow--Scholze~\cite{bhattmorrowscholze2} have constructed
weight\footnote{\,If $S$ is a smooth $\mathbb{F}_p$-algebra, then, on
the $j$th graded piece of the Bhatt--Morrow--Scholze filtration on
$\TP_*(S,\mathbb{Z}_p)[1/p]$, the geometric Frobenius
$\operatorname{Fr}_p$ acts with pure weight $j$ in the
sense that $\smash{ \operatorname{Fr}_p^* = p^j\varphi_p }$, where
$\varphi_p$ is the cyclotomic Frobenius.}
filtrations\index{Bhatt--Morrow--Scholze filtration} of topological
cyclic homology and its variants such that, on $j$th graded pieces,
the equalizer of $p$-completed spectra
$$\begin{xy}
(0,0)*+{ \TC(S,\mathbb{Z}_p) }="1";
(25,0)*+{ \TC^{-}(S,\mathbb{Z}_p) }="2";
(52,0)*+{ \TP(S,\mathbb{Z}_p) }="3";
{ \ar "2";"1";};
{ \ar@<.7ex>^-{\varphi} "3";"2";};
{ \ar@<-.7ex>_-{\operatorname{can}} "3";"2";};
\end{xy}$$
gives rise to an equalizer
\vspace{-2mm}
$$\begin{xy}
(0,0)*+{\phantom{\widehat{\mathbbl{\Delta}}_S} \mathbb{Z}_p(j)[2j] }="1";
(32,0)*+{ \operatorname{Fil}^j\widehat{\mathbbl{\Delta}}_S\{j\}[2j] }="2";
(66,0)*+{ \widehat{\mathbbl{\Delta}}_S\{j\}[2j]
\phantom{\widehat{\mathbbl{\Delta}}_S} }="3";
{ \ar "2";"1";};
{ \ar@<.7ex>^-{\text{``$\frac{\varphi}{\xi^j}$''}} "3";"2";};
{ \ar@<-.7ex>_-{\text{``$\operatorname{incl}$''}} "3";"2";};
\end{xy}$$
Here $S$ is any commutative ring, $\widehat{\mathbbl{\Delta}}_S =
\widehat{\mathbbl{\Delta}}_S\{0\}$ is an $\mathbb{E}_{\infty}$-algebra
in the derived $\infty$-category of $S$-modules,
$\widehat{\mathbbl{\Delta}}_S\{j\}$ is an invertible
$\widehat{\mathbbl{\Delta}}_S$-module, and
$\operatorname{Fil}^{\boldsymbol{\cdot}}\widehat{\mathbbl{\Delta}}_S\{j\}$
is the derived complete descending ``Nygaard'' filtration thereof. The
equalizer $\mathbb{Z}_p(j)$ is a version of syntomic
cohomology\index{syntomic cohomology} that works correctly for all
weights $j$, as opposed to only for $j < p-1$. The
Bhatt--Morrow--Scholze filtration of $\TP(S,\mathbb{Z}_p)$ gives rise
to an Atiyah--Hirzebruch type spectral sequence
$$E_{i,j}^2 = H^{j-i}(\operatorname{Spec}(S),\widehat{\mathbbl{\Delta}}_S\{j\})
\Rightarrow \TP_{i+j}(S,\mathbb{Z}_p),$$
and similarly for $\TC(S,\mathbb{Z}_p)$ and
$\TC^{-}(S,\mathbb{Z}_p)$. If $R$ is perfectoid, then
$$H^i(\operatorname{Spec}(R),\widehat{\mathbbl{\Delta}}_R\{j\}) \simeq
\begin{cases}
A_{\operatorname{inf}}(R) & \text{if $i = 0$} \cr
0 & \text{if $i \neq 0$,} \cr
\end{cases}$$
for all integers $j$, and methods for evaluating these ``prismatic''
cohomology\index{prismatic cohomology} groups are currently being
developed by Bhatt--Scholze~\cite{bhattscholze}. 

It is a great pleasure to acknowledge the  support that we have
received while preparing this chapter. Hesselholt was funded in part by
the Isaac Newton Institute as a Rothschild Distinguished Visiting
Fellow and by the Mathematical Sciences Research Institute as a Simons
Visiting Professor, and Nikolaus was funded in part by Deutsche
Forschungsgemeinschaft under Germany's Excellence Strategy EXC 2044 D390685587, Mathematics M\"unster: Dynamics--Geometry--Structure.

\section{Topological Hochschild homology}\label{sec:thh}

We sketch the definition of topological Hochschild homology,
topological cyclic homology, and the cyclotomic trace map from
algebraic $K$-theory following
Nikolaus--Scholze~\cite{nikolausscholze} and Nikolaus~\cite{nikolaus}.

\subsection{Definition}
\index{topological Hochschild homology|(}

If $R$ is an $\mathbb{E}_\infty$-algebra in spectra, then we define
$\THH(R)$ to be the colimit of the diagram $\mathbb{T} \to
\operatorname{Alg}_{\mathbb{E}_{\infty}}(\mathsf{Sp})$ that is
constant with value $R$, and we write
$$\THH(R) = R^{\hskip.5pt\otimes\hskip.5pt\mathbb{T}}$$
to indicate this colimit. Here $\mathbb{T}$
is the circle group. The action of $\mathbb{T}$ on itself by left
translation induces a $\mathbb{T}$-action on the
$\mathbb{E}_{\infty}$-algebra $\THH(R)$. In addition, the map of
$\mathbb{E}_{\infty}$-algebras $R \to \THH(R)$ induced by the
structure map of the colimit exhibits $\THH(R)$ as the initial
$\mathbb{E}_{\infty}$-algebra with $\mathbb{T}$-action under $R$.
\index{topological Hochschild homology!of
$\mathbb{E}_{\infty}$-algebra in spectra}

We let $p$ be a prime number, and let $C_p \subset \mathbb{T}$ be
the subgroup of order $p$. The
\emph{Tate diagonal}\index{Tate diagonal} is a natural map of
$\mathbb{E}_{\infty}$-algebras in spectra
$$\xymatrix{
{ R } \ar[r]^-{\Delta_p} &
{ (R^{\otimes C_p})^{tC_p}. } \cr
}$$
Heuristically, this map takes $a$ to the equivalence class
of $a \hskip.5pt\otimes \dots \otimes\hskip.5pt a$, but it exist only
in higher algebra.\footnote{\,In fact, if $k$ is a commutative ring,
  then the space of natural transformations between the corresponding
  endofunctors on $\operatorname{Alg}_{\mathbb{E}_{\infty}}(\mathcal{D}(k))$ is empty.}
Moreover, the map $R \to \THH(R)$ of
$\mathbb{E}_{\infty}$-algebras in spectra extends uniquely to a map
$R^{\otimes C_p} \to \THH(R)$ of $\mathbb{E}_{\infty}$-algebras in spectra
with $C_p$-action, which, in turn, induces a map of Tate spectra
$$\xymatrix{
{ (R^{\otimes C_p})^{tC_p} } \ar[r] &
{ \THH(R)^{tC_p}. } \cr
}$$
This map also is a map of $\mathbb{E}_{\infty}$-algebras, and its
target carries a residual action of $\mathbb{T}/C_p$, which we
identify with $\mathbb{T}$ via the isomorphism given by the $p$th
root. Hence, by the universal property of $R \to \THH(R)$, there
is a unique map $\varphi_p$ of $\mathbb{E}_{\infty}$-algebras in
spectra with $\mathbb{T}$-action which makes the diagram
$$\xymatrix{
{ R } \ar[r]^-{\Delta_p} \ar[d] &
{ (R^{\otimes C_p})^{tC_p} } \ar[d] \cr
{ \THH(R) } \ar[r]^-{\varphi_p} &
{ \THH(R)^{tC_p} } \cr
}$$
in
$\operatorname{Alg}_{\mathbb{E}_{\infty}}(\operatorname{Sp})$ commute
(in the $\infty$-categorical sense). The map $\varphi_p$ is called the
$p$th cyclotomic Frobenius,\index{cyclotomic Frobenius} and the family
of maps $(\varphi_p)_{p \in   \mathbb{P}}$ indexed by the set
$\mathbb{P}$ of prime numbers makes $\THH(R)$ a \emph{cyclotomic}
spectrum in the following sense.

\begin{definition}[Nikolaus--Scholze]\label{def:cyclotomicspectrum}
\index{cyclotomic spectrum}A cyclotomic spectrum is a pair of a
spectrum with $\mathbb{T}$-action $X$ and a family
$(\varphi_p)_{p \in \mathbb{P}}$ of $\mathbb{T}$-equivariant maps
$$\xymatrix{
{ X } \ar[r]^-{\varphi_p} &
{ X^{tC_p}. } \cr
}$$
The $\infty$-category of cyclotomic spectra\index{$\infty$-category of
  cyclotomic spectra} is the pullback of simplicial sets
$$\xymatrix{
{ \operatorname{CycSp} } \ar[rr] \ar[d] &&
{ \prod_{p \in \mathbb{P}}
  \operatorname{Fun}(\Delta^1,\operatorname{Sp}^{B\mathbb{T}}) }
  \ar[d]^{\prod (\mathrm{ev}_0, \mathrm{ev}_1)} \cr
{ \operatorname{Sp}^{B\mathbb{T}} }
\ar[rr]^-{(\id, (-)^{tC_p})_{p \in \mathbb{P}}} &&
{ \prod_{p \in \mathbb{P}} \operatorname{Sp}^{B\mathbb{T}} \times
  \operatorname{Sp}^{B\mathbb{T}}. } \cr
}$$
\end{definition}

We remark that, in contrast to the earlier notions of cyclotomic spectra in
Hesselholt--Madsen~\cite{hm} and
Blumberg--Mandell~\cite{blumbergmandell1}, the Nikolaus--Scholze
definition does not require equivariant homotopy theory.

It is shown in \cite{nikolausscholze} that $\operatorname{CycSp}$ is a
presentable stable $\infty$-category, and that it canonically extends
to a symmetric monoidal $\infty$-category
$$\xymatrix{
{ \operatorname{CycSp}^{\otimes } } \ar[r] &
{ \operatorname{Fin}_* } \cr
}$$
with underlying $\infty$-category
$\operatorname{CycSp}$. Now, the construction of $\THH$ given above
produces a lax symmetric monoidal functor
$$\begin{xy}
(0,0)*+{
  \operatorname{Alg}_{\mathbb{E}_{\infty}}(\operatorname{Sp}^{\otimes})
}="1";
(35,0)*+{
  \operatorname{Alg}_{\mathbb{E}_{\infty}}(\operatorname{CycSp}^{\otimes}). }="2";
{ \ar^-{\THH} "2";"1";};
\end{xy}$$
Topological Hochschild homology may be defined, more generally, for
(small) stable $\infty$-categories $\mathcal{C}$.
\index{topological Hochschild homology!of stable $\infty$-categories}
If $R$ is an $\mathbb{E}_{\infty}$-algebra in spectra and $\mathcal{C} =
\operatorname{Perf}_R$ is the stable $\infty$-category of perfect
$R$-modules, there is a canonical equivalence
$$\THH(R) \simeq \THH(\operatorname{Perf}_R)$$
of cyclotomic spectra. 
The basic idea is to define the underlying spectrum
with $\mathbb{T}$-action $\THH(\mathcal{C})$ to be the geometric
realization of the cyclic spectrum that, in simplicial degree $n$, is
given by
$$\THH(\mathcal{C})_n = \operatorname*{colim} \Big(\bigotimes_{0 \leq i \leq n}
\operatorname{map}_{\mathcal C}(x_i , x_{i+1})\Big),$$
where the colimit ranges over the space of $(n+1)$-tuples in the groupoid
core~$\mathcal{C}^{\simeq}$ of the $\infty$-category $\mathcal{C}$,
$\mathrm{map}_{\mathcal C}$ denotes 
the mapping spectrum in $\mathcal C$, and the index $i$ is taken modulo
$n+1$. We indicate the steps necessary to make sense of this
definition, see~\cite{nikolaus} and the forthcoming
paper~\cite{nikolaus1} for details.

First, to make sense of the colimit above, one must construct a functor 
$$\xymatrix{
{ (\mathcal{C}^{\simeq})^{n+1} } \ar[r] &
{ \operatorname{Sp} } \cr
}$$
that to $(x_0, \dots, x_n)$ assigns $\bigotimes_{0 \leq i \leq n}
\operatorname{map}_{\mathcal C}(x_i , x_{i+1})$.
This can be achieved by a combination of the tensor product functor,
the mapping spectrum functor
$\operatorname{map}_{\mathcal{C}} \colon
\mathcal{C}^{\operatorname{op}} \times
\mathcal{C} \to \operatorname{Sp}$, and the canonical equivalence
$\mathcal{C}^\simeq \simeq
(\mathcal{C}^{\operatorname{op}})^{\simeq}$. Second, one must lift the
assignment $n \mapsto \THH(\mathcal{C})_n$ to a functor
$\Lambda^{\operatorname{op}} \to \operatorname{Sp}$ from Connes'
cyclic category such that the face and degeneracy maps are given by
composing adjacent morphisms and by inserting identities,
respectively, while the cyclic operator is given by cyclic permutation
of the tensor factors. As explained
in~\cite[Appendix~B]{nikolausscholze}, for every cyclic spectrum
$\Lambda^{\mathrm{op}} \to \operatorname{Sp}$, the geometric
realization of the simplicial spectrum
$\Delta^{\mathrm{op}} \to \Lambda^{\mathrm{op}} \to \operatorname{Sp}$
carries a natural $\mathbb{T}$-action. Finally, one must construct the
cyclotomic Frobenius maps
$$\xymatrix{
{ \THH(\mathcal{C}) } \ar[r]^-{\varphi_p} &
{ \THH(\mathcal{C})^{tC_p}. } \cr
}$$
These are defined following~\cite[Section~III.2]{nikolausscholze} as the
Tate-diagonal applied levelwise followed by the canonical colimit-Tate
interchange map.
\index{topological Hochschild homology|)}

\subsection{Topological cyclic homology and the trace}
\index{topological cyclic homology|(}

Taking $R$ to be the sphere spectrum $\mathbb{S}$, we obtain an
$\mathbb{E}_{\infty}$-algebra in cyclotomic spectra
$\THH(\mathbb{S})$, which we denote by
$\mathbb{S}^{\,\operatorname{triv}}$. Its underlying spectrum is 
$\mathbb{S}$, and its cyclotomic Frobenius map $\varphi_p$ can be identified with a canonical $\mathbb{T}$-equivariant refinement of the
composition
$$\xymatrix{
{ \mathbb{S} } \ar[r]^-{\operatorname{triv}} &
{ \mathbb{S}^{hC_p} } \ar[r]^-{\operatorname{can}} &
{ \mathbb{S}^{tC_p} } \cr
}$$
of the map $\operatorname{triv}$ induced from the projection
$BC_p \to \operatorname{pt}$ and the canonical map. Since the
$\infty$-category $\operatorname{CycSp}$ is stable, we have associated
with every pair of objects $X, Y \in \operatorname{CycSp}$ a mapping
spectrum $\operatorname{map}_{\operatorname{CycSp}}(X,Y)$, which depends
functorially on $X$ and $Y$. 

\begin{definition}\label{def:tc}The topological cyclic homology of a cyclotomic
spectrum $X$ is the mapping spectrum
$\TC(X) =
\operatorname{map}_{\operatorname{CycSp}}(\mathbb{S}^{\operatorname{triv}},X)$.
\end{definition}

If $X = \THH(R)$ with $R$ an $\mathbb{E}_{\infty}$-algebra in spectra,
then we abbreviate and write $\TC(R)$ instead of $\TC(X)$. Similarly,
if $X = \THH(\mathcal{C})$ with $\mathcal{C}$ a stable
$\infty$-category, then we write $\TC(\mathcal{C})$ instead of
$\TC(X)$.

This definition of topological cyclic homology as given above is
abstractly elegant and useful, but for concrete calculations, a more
concrete formula  is necessary. Therefore, we unpack
Definition~\ref{def:tc}. We assume that $X$ is bounded below and
write
\index{topological cyclic homology!negative}
\index{negative topological cyclic homology}
\index{topological cyclic homology!periodic}
\index{periodic topological cyclic homology}
$$\xymatrix{
{ \TC^-(X) } \ar[r]^-{\operatorname{can}} &
{ \TP(X) } \cr
}$$
for the canonical map $X^{h\mathbb{T}} \to X^{t\mathbb{T}}$. We call
the domain and target of this map the negative topological cyclic
homology and the periodic topological cyclic homology of $X$,
respectively. Since the cyclotomic Frobenius maps
$$\xymatrix{
{ X } \ar[r]^-{\varphi_p} &
{ X^{tC_p} } \cr
}$$
are $\mathbb{T}$-equivariant, they give rise to a map of
$\mathbb{T}$-homotopy fixed points spectra
$$\begin{xy}
(0,0)*+{ \phantom{ \prod_{p \in \mathbb{P}} } X^{h\mathbb{T}} }="1";
(38,0)*+{ \prod_{p \in \mathbb{P}} (X^{tC_p})^{h\mathbb{T}}.
   \phantom{ \prod_{p \in \mathbb{P}} }  }="2";
{ \ar^-{(\varphi_p^{h\mathbb{T}})} "2";"1";};
\end{xy}$$
Moreover, there is a canonical map
$$\xymatrix{
{ \prod_{p \in \mathbb{P}} (X^{tC_p})^{h\mathbb{T}} } &
{ X^{t\mathbb{T}}, } \ar[l] \cr
}$$
which becomes an equivalence after profinite completion, since $X$ is
bounded below, by the Tate-orbit lemma~\cite[Lemma
II.4.2]{nikolausscholze}.\index{Tate orbit lemma} Hence, we get a map
$$\xymatrix{
{ \TC^{-}(X) } \ar[r]^-{\varphi} &
{ \TP(X)^\wedge, } \cr
}$$
where ``$(-)^{\wedge}$'' indicates profinite completion. There is also
a canonical map  
$$\xymatrix{
{ \TC^-(X) } \ar[r]^-{\operatorname{can}} &
{ \TP(X)^\wedge } \cr
}$$
given by the composition of the canonical from the homotopy fixed
point spectrum to the Tate construction followed by the completion
map. This gives the following description of $\TC(X)$.

\begin{proposition}\label{prop:tcformula}For bounded below cyclotomic
spectra $X$, there is natural equalizer diagram
\vspace{-2mm}
$$\xymatrix{
{ \TC(X) } \ar[r] &
{ \TC^-(X) } \ar[r]<.7ex>^-{\varphi}
\ar[r]<-.7ex>_-{\operatorname{can}} &
{ \TP(X)^\wedge. } \cr
}$$
\end{proposition}

We now explain the definition of the cyclotomic trace map from
$K$-theory to topological cyclic
homology.\index{cyclotomic trace map} Let
$\operatorname{Cat}_{\infty}^{\operatorname{stab}}$ be the 
$\infty$-category of small, stable $\infty$-categories and 
exact functors. The $\infty$-category of noncommutative
motives\index{noncommutative motives} (or a
slight variant thereof) of
Blumberg--Gepner--Tabuada~\cite{blumberggepnertabuada} is defined to
be the initial (large) $\infty$-category with a functor
$$\xymatrix{
z : { \operatorname{Cat}_{\infty}^{\operatorname{stab}}} \ar[r] &
{ \operatorname{NMot} } \cr
}$$
such that the following hold:
\begin{enumerate}
\item[(1)]\label{eins}(Stability)
The $\infty$-category $\operatorname{NMot}$ is stable.
\item[(2)]\label{zwei}(Localization)
For every Verdier sequence\footnote{\,This
  means that $\mathcal{C} \to \mathcal{D} \to \mathcal{C}/\mathcal{D}$
  is both a fiber sequence and cofiber sequence in
  $\operatorname{Cat}_\infty^{\operatorname{stab}}$. In particular,
  $\mathcal{D} \to \mathcal{C}$ is fully faithful and its image in
  $\mathcal{D}$ is closed under retracts.}
$\mathcal{C} \to \mathcal{D} \to \mathcal{D}/\mathcal{C}$ in
$\operatorname{Cat}_{\infty}^{\operatorname{stab}}$, the image sequence
$z(\mathcal{C}) \to z(\mathcal{D}) \to z(\mathcal{D}/\mathcal{C})$
in $\operatorname{NMot}$ is a fiber sequence.
\item[(3)]\label{drei}(Morita invariance)
For every map $\mathcal{C} \to \mathcal{D}$ in
$\operatorname{Cat}_{\infty}^{\operatorname{stab}}$ that becomes an
equivalence after idempotent completion, the image map
$z(\mathcal{C}) \to z(\mathcal{D})$ in $\operatorname{NMot}$ is an
equivalence.
\end{enumerate}
The main theorem of
op.~cit.~states\footnote{\,In fact, we do not require $z \colon
\operatorname{Cat}_{\infty}^{\operatorname{stab}} \to
\operatorname{NMot}$ to preserve filtered colimits, as
do~\cite{blumberggepnertabuada}.} that for every (small) stable
$\infty$-category $\mathcal{C}$, there is a canonical equivalence
$$K(\mathcal{C}) \simeq
\operatorname{map}_{\operatorname{NMot}}(z(\operatorname{Perf}_{\mathbb{S}}),z(\mathcal{C}))$$
between its nonconnective algebraic $K$-theory spectrum and the indicated
mapping spectrum in $\operatorname{NMot}$. In general, one may
view the mapping spectra in $\operatorname{NMot}$ as bivariant
versions of nonconnective algebraic $K$-theory. Accordingly, the
mapping spectra in  $\operatorname{CycSp}$ are bivariant versions of
$\TC$.\index{algebraic $K$-theory!of stable $\infty$-category} As we
outlined in the previous section, topological Hochschild homology 
is a functor
$$\begin{xy}
(0,0)*+{ \operatorname{Cat}_{\infty}^{\operatorname{stab}} }="1";
(25,0)*+{ \operatorname{CycSp}, }="2";
{ \ar^-{\THH} "2";"1";};
\end{xy}$$
and one can show that it satisfies the properties~(1)--(3) above. There is
a very elegant proof of~(2) and~(3) based on work of Keller,
Blumberg--Mandell, and Kaledin that uses the trace property of
$\THH$, see the forthcoming  paper~\cite{nikolaus1} for a summary.
\index{topological Hochschild homology!trace property of}
Accordingly, the functor $\THH$ admits a unique factorization
$$\xymatrix{
{ \operatorname{Cat}_{\infty}^{\operatorname{stab}} } \ar[r]^-{z} &
{ \operatorname{NMot} } \ar[r]^-{\tr} &
{ \operatorname{CycSp} }
}$$
with $\tr$ exact. In particular, for every stable $\infty$-category
$\mathcal{C}$, we have an induced map of mapping spectra
$$\xymatrix{
{
  \operatorname{map}_{\operatorname{NMot}}(z(\operatorname{Perf}_{\mathbb{S}}),
  z(\mathcal{C})) } \ar[r]^-{\tr} &
{
  \operatorname{map}_{\operatorname{CycSp}}(\mathbb{S}^{\,\operatorname{triv}},
  \THH(\mathcal{C})). } \cr
}$$
This map, by definition, is the cyclotomic trace map, which we write
$$\xymatrix{
{ K(\mathcal{C}) } \ar[r]^-{\tr} &
{ \TC(\mathcal{C}). } \cr
}$$
More concretely, on connective covers, considered here as
$\mathbb{E}_\infty$-groups in spaces, the cyclotomic trace map is
given by the composition
\index{cyclotomic trace map!on connective covers}
$$\begin{aligned}
{ \Omega^{\infty} K(\mathcal{C}) }
& \simeq \Omega ( \operatorname{colim}_{\Delta^{\operatorname{op}}}
(S\mathcal{C}^{\operatorname{idem}})^{\simeq} ) 
\to \Omega ( \operatorname{colim}_{\Delta^{\operatorname{op}}}
\Omega^{\infty}\TC( S\mathcal{C}^{\operatorname{idem}}) ) \cr
{} & \to \Omega^{\infty}\Omega (
\operatorname{colim}_{\Delta^{\operatorname{op}}} \TC(
S\mathcal{C}^{\operatorname{idem}} ) )
\simeq \Omega^{\infty} \TC(\mathcal{C}), \cr
\end{aligned}$$
where $S(-)$ and $(-)^{\mathrm{idem}}$ indicate Waldhausen's
construction and idempotent completion, respectively, where the
second map is induced from the map   
$$\xymatrix{
{ \mathcal{D}^{\simeq} \simeq
\operatorname{Map}_{\operatorname{Cat}_{\infty}^{\operatorname{stab}}}(
\operatorname{Perf}_{\mathbb{S}} ,\mathcal{D}) } \ar[r]^-{\THH} &
{ \operatorname{Map}_{\operatorname{CycSp}}(
\mathbb{S}^{\operatorname{triv}}, \THH(\mathcal{D}) ) 
\simeq \Omega^\infty\TC(\mathcal{D}), } \cr
}$$
and where last equivalence follows from $\TC$ satisfying (2) and~(3) above. 
\index{topological cyclic homology|)}

\subsection{Connes' operator}\label{sec:connesoperator}
\index{Connes' operator}

The symmetric monoidal $\infty$-category of spectra with
$\mathbb{T}$-action is canonically equivalent to the symmetric
monoidal $\infty$-category of modules over the group algebra
$\mathbb{S}[\mathbb{T}]$. The latter is an
$\mathbb{E}_{\infty}$-algebra in spectra, and
$$\pi_*(\mathbb{S}[\mathbb{T}]) = (\pi_*\mathbb{S})[d]/(d^2 - \eta d),$$
where $d$ has degree $1$ and is obtained from a choice of generator
of $\pi_1(\mathbb{T})$ by translating it to the basepoint in the group
ring. The relation $d^2 =\eta d$ is a consequence of the fact that,
stably, the multiplication map
$\mu \colon \mathbb{T} \times \mathbb{T} \to \mathbb{T}$ splits off
the Hopf map $\eta \in \pi_1(\mathbb{S})$. From this calculation we
conclude that a $\mathbb{T}$-action on an
$\mathbb{E}_{\infty}$-algebra in spectra $T$ gives rise to a graded
derivation
$$\xymatrix{
{ \pi_j(T) } \ar[r]^-{d} &
{ \pi_{j+1}(T), } \cr
}$$
and this is Connes' operator. The operator $d$ is not quite a
differential, since we have $d \circ d = d \circ \eta = \eta \circ d$.

The $\mathbb{E}_{\infty}$-algebra structure on $T$ gives rise to power
operations in homology. In singular homology with
$\mathbb{F}_2$-coefficients, there are power operations
\index{power operations!Araki--Kudo}
$$\xymatrix{
{ \pi_j(\mathbb{F}_2 \otimes T) } \ar[r]^-{Q^i} &
{ \pi_{i+j}(\mathbb{F}_2 \otimes T) } \cr
}$$
for all integers $i$ introduced by Araki--Kudo~\cite{arakikudo}, and
in singular homology with $\mathbb{F}_p$-coefficients, where $p$ is
odd, there are similar power operations
\index{power operations!Dyer--Lashof}
$$\xymatrix{
{ \pi_j(\mathbb{F}_p \otimes T) } \ar[r]^-{Q^i} &
{ \pi_{2i(p-1)+j}(\mathbb{F}_p \otimes T) } \cr
}$$
for all integers $i$ defined by Dyer--Lashof~\cite{dyerlashof}. The
power operations are natural with respect to maps of
$\mathbb{E}_{\infty}$-rings, but it is not immediately clear that they
are compatible with Connes' operator, too. We give a proof that this
is the nevertheless the case, following 
Angeltveit--Rognes~\cite[Proposition~5.9]{angeltveitrognes} and the
very nice exposition of H\"{o}ning~\cite{honing}.

\begin{proposition}\label{prop:connesdyerlashof}
\index{Connes' operator!and power operations}
If $T$ is an $\mathbb{E}_{\infty}$-ring with $\mathbb{T}$-action, then
$$Q^i \circ d = d \circ Q^i$$
for all integers $i$. 
\end{proposition}

\begin{proof}The adjunct
$\tilde{\mu} \colon T \to \map(\mathbb{T}_+,T)$ of the
map $\mu \colon T \otimes \mathbb{T}_+ \to T$ induced by the
$\mathbb{T}$-action on $T$ is a map of $\mathbb{E}_{\infty}$-rings, as
is the canonical equivalence $T \otimes \map(\mathbb{T}_+,\mathbb{S})
\to \map(\mathbb{T}_+,T)$. Composing the former map with an
inverse of the latter map, we obtain a map of $\mathbb{E}_{\infty}$-rings
$$\xymatrix{
{ T } \ar[r]^-{\tilde{d}} &
{ T \otimes \map(\mathbb{T}_+,\mathbb{S}), } \cr
}$$
and hence, the induced map on homology preserves power operations. We
identify the homology of the target via the isomorphism
$$\xymatrix{
{ \pi_*(\mathbb{F}_p \otimes T) \otimes_{\pi_*(\mathbb{F}_p)}
  \pi_*(\mathbb{F}_p \otimes \map(\mathbb{T}_+,\mathbb{S})) } \ar[r]
&
{ \pi_*(\mathbb{F}_p \otimes T \otimes \map(\mathbb{T}_+,\mathbb{S})) } \cr
}$$
and the direct sum decomposition of $\map(\mathbb{T}_+,\mathbb{S})$
induced by the direct sum decomposition above, and under this
idenfication, we have
$$\tilde{d}(a) = a \otimes 1 + da \otimes [S^{-1}].$$
Now, the Cartan formula for power operations shows that
$$Q^i(\tilde{d}(a)) = Q^i(a \otimes 1 + da \otimes [S^{-1}])
= Q^i(a) \otimes 1 + Q^i(da) \otimes [S^{-1}],$$
since $Q^0([S^{-1}]) = [S^{-1}]$ and $Q^i([S^{-1}]) = 0$ for
$i \neq 0$. But we also have
$$Q^i(\tilde{d}(a)) = \tilde{d}(Q^i(a)) = Q^i(a) \otimes 1 + d(Q^i(a))
\otimes [S^{-1}],$$
so we conclude that $Q^i \circ d = d \circ Q^i$ as stated.
\end{proof}

We finally discuss the HKR-filtration. If $k$ is a commutative ring
and $A$ a simplicial commutative $k$-algebra, then the Hochschild
spectrum\index{Hochschild homology}
\index{Hochschild--Kostant--Rosenberg filtration}
$$\HH(A/k) = A^{\otimes_k \mathbb{T}}$$
has a complete and $\mathbb{T}$-equivariant descending filtration
$$\cdots \subseteq \Fil^n\HH(A/k) \subseteq \cdots \subseteq \Fil^1\HH(A/k)
\subseteq \Fil^0\HH(A/k) \simeq \HH(A/k)$$
defined as follows. If $A/k$ is smooth and discrete, then
$$\Fil^n\HH(A/k) = \tau_{\geq n}\HH(A/k),$$
and in general, the filtration is obtained from this special case by
left Kan extension. The filtration quotients are identified
as follows. If $A/k$ is discrete, then $\HH(A/k)$ may be represented
by a simplicial commutative $A$-algebra, and hence, its homotopy groups
$\HH_*(A/k)$ form a strictly\footnote{\,Here ``strictly'' indicates
that elements of odd degree square to zero. This follows
from~\cite[Th\'{e}or\`{e}me~4]{cartan} by considering the universal
case of Eilenberg--MacLane spaces.}anticommutative graded
$A$-algebra. Moreover, Connes' 
operator gives rise to a differential on $\HH_*(A/k)$, which raises
degrees by one and is a graded $k$-linear derivation. By definition,
the de~Rham-complex $\Omega_{A/k}^*$ is the universal example of this
algebraic structure, and therefore, we have a canonical map 
$$\xymatrix{
{ \Omega_{A/k}^* } \ar[r] &
{ \HH_*(A/k),} \cr
}$$
which, by the Hochschild--Kostant--Rosenberg theorem, is an
isomorphism, if $A/k$ smooth. By the definition of the cotangent
complex, this shows that
$$\gr^j\HH(A/k) \simeq (\Lambda_A^j L_{A/k})[j]$$
with trivial $\mathbb{T}$-action. Here $\Lambda^j_A$ indicates the
non-abelian derived functor of the $j$th exterior power over $A$. 

\section{B\"{o}kstedt periodicity}\label{sec:bokstedtperiodicity}

B\"{o}kstedt periodicity is the fundamental result that
$\THH_*(\mathbb{F}_p)$ is a polynomial algebra over $\mathbb{F}_p$ on
a generator in degree two. We present a proof, which is close to
B\"{o}kstedt's original proof in the unpublished
manuscript~\cite{bokstedt1}. The skeleton filtration of the standard
simplicial model for the circle induces a filtration
of the topological Hochschild spectrum. For every homology theory, this
gives rise to a spectral sequence, called the B\"{o}kstedt spectral
sequence, that converges to the homology of the topological Hochschild
spectrum. It is a spectral sequence of Hopf algebras in the symmetric monoidal
category of quasi-coherent sheaves on the stack defined by the
homology theory in question, and  to handle this rich algebraic
structure, we find it useful to introduce the geometric language of
Berthelot and Grothendieck~\cite[II.1]{berthelot}.

\subsection{The Adams spectral sequence}
\index{Adams spectral sequence|(}
If $f \colon A \to B$ is a map of anticommutative graded rings, then
extension of scalars\index{extension of scalars} along $f$ and
restriction of scalars\index{restriction of scalars} along $f$ define
adjoint functors
$$\xymatrix{
{ \operatorname{Mod}_A } \ar@<.7ex>[r]^-(.53){f^*} &
{ \operatorname{Mod}_B } \ar@<.7ex>[l]^-(.47){f_*} \cr
}$$
between the respective categories of graded modules. Moreover, the
extension of scalars functor $f^*$ is symmetric monoidal, while the
restriction of scalars functor $f_*$ is lax symmetric monoidal with
respect to the tensor product of graded modules.

We let $k$ be an $\mathbb{E}_{\infty}$-ring and form the
cosimplicial $\mathbb{E}_{\infty}$-ring
$$\begin{xy}
(0,0)*+{ {}_{\phantom{\mathbb{E}_{\infty}}} \Delta }="1";
(30,0)*+{ \operatorname{Alg}_{\mathbb{E}_{\infty}}(\mathsf{Sp}).
  {}_{\phantom{\mathbb{E}_{\infty}}}  }="2";
{ \ar^-{k^{\otimes[-]}} "2";"1";};
\end{xy}$$
Here, as usual, $[n] \in \Delta$ denotes the finite ordinal
$\{0,1,\dots,n\}$, so $k^{\otimes [n]}$ is an $(n+1)$-fold tensor
product. We will assume 
that the map
$$\xymatrix{
{ A = \pi_*(k^{\otimes[0]}) } \ar[r]^-{d^1} &
{ \pi_*(k^{\otimes[1]}) = B } \cr
}$$
is flat, so that
$d^{\hskip.5pt 0},d^{\hskip.5pt 2} \colon k^{\otimes [1]} \to k^{\otimes [2]}$
induce an isomorphism of graded rings
$$\xymatrix{
{ B \otimes_A \hskip-1pt B = \pi_*(k^{\otimes [1]})
  \otimes_{\pi_*(k^{\otimes [0]})}
  \pi_*(k^{\otimes [1]}) } \ar[r]^-{d^{\hskip.5pt 2}+d^{\hskip.5pt 0}} &
{ \pi_*(k^{\otimes [2]}). } \cr
}$$
The map $d^1 \colon k^{\otimes [1]} \to k^{\otimes [2]}$ now gives
rise to a map of graded rings
$$\xymatrix{
{ B } \ar[r]^-{d^1} &
{ B \otimes_AB }
}$$
and the sextuple
$$(A,B, \xymatrix{
{ B } \ar[r]^-{s^0} &
{ A } \cr
}, 
\xymatrix{
{ A } \ar@<.7ex>[r]^-{d^{\hskip.5pt 0}} \ar@<-.7ex>[r]_-{d^1} &
{ B } \cr
},
\xymatrix{
{B } \ar[r]^-{d^1} &
{ B \otimes_A \hskip-1pt B} \cr
})$$
forms a cocategory object in the category of graded rings with the
cocartesian symmetric monoidal structure. Here the maps
$s^0 \colon A \to B$, $d^{\hskip.5pt 0},d^1 \colon A \to B$, and
$d^1 \colon B \to B \otimes_AB$ are the opposites of the unit map, the
source and target maps, and the composition map. Likewise, the
septuple, where we also include the map $\chi \colon B \to B$ induced
by the unique non-identity automorphism of the set $[1] = \{0,1\}$,
forms a
cogroupoid\index{cogroupoid}\index{cogroupoid!associated with
$\mathbb{E}_{\infty}$-ring} in this symmetric monoidal category. We
will abbreviate and simply write $(A,B)$ for this cogroupoid object.

In general, given a cogroupoid object $(A,B)$ in graded rings, we define
an\index{cogroupoid module}
$(A,B)$-module\footnote{\,The cogroupoid $(A,B)$ defines a stack
$\mathcal{X}$, and the categories of $(A,B)$-modules and
quasi-coherent $\mathcal{O}_{\mathcal{X}}$-modules are
equivalent. For this reason, we prefer to say $(A,B)$-module instead
of $(A,B)$-comodule, as is more common in the homotopy theory
literature.} to be a pair $(M,\epsilon)$ of an $A$-module $M$ and a
$B$-linear map
$$\xymatrix{
{ d^{1*}(M) } \ar[r]^-{\epsilon} &
{ d^{\hskip.5pt 0*}(M) } \cr
}$$
that makes the following diagrams, in which the equality signs
indicate the unique isomorphisms, commute.
$$\begin{xy}
(0,5)*+{ s^{0*}d^{1*}(M) }="11";
(32,5)*+{ s^{0*}d^{\hskip.5pt 0*} (M) }="13";
(16,-5)*+{ M }="22";
{ \ar^-{s^{0*}(\epsilon)} "13";"11";};
{ \ar@{=} "22";"11";};
{ \ar@{=} "22";"13";};
\end{xy}$$
$$\begin{xy}
(13,12)*+{ d^{1*}d^{1*}(M) }="11";
(48,12)*+{ d^{1*}d^{\hskip.5pt 0*}(M) }="12";
 (0,0)*+{ (B \otimes d^1)^*d^{1*}(M) }="21";
(61,0)*+{ (d^{\hskip.5pt 0} \otimes B)^*d^{\hskip.5pt 0*}(M) }="22";
(13,-12)*+{ (B \otimes d^1)^*d^{\hskip.5pt 0*}(M) }="31";
(48,-12)*+{ (d^{\hskip.5pt 0} \otimes B)^*d^{1*}(M) }="32";
{ \ar^-{d^{1*}(\epsilon)} "12";"11";};
{ \ar_-(.1){(B \otimes d^1)^*(\epsilon)} "31";"21";};
{ \ar_-(.9){(d^{\hskip.5pt 0} \otimes B)^*(\epsilon)} "22";"32";};
{ \ar@{=} "11";"21";};
{ \ar@{=} "22";"12";};
{ \ar@{=} "32";"31";};
\end{xy}$$
We say that $\epsilon$ is a stratification of the $A$-module
$M$ relative to the cogroupoid $(A,B)$.\index{stratification!relative
to cogroupoid} The map $\epsilon$, we
remark, is necessarily an isomorphism. We define a map of
$(A,B)$-modules $f \colon (M_0,\epsilon_0) \to (M_1,\epsilon_1)$ to be
a map of $A$-modules $f \colon M_0 \to M_1$ that makes the diagram of
$B$-modules
$$\xymatrix{
{ d^{1*}(M_0) } \ar[r]^-{\epsilon_0}
\ar[d]^-{d^{\hskip.5pt 0*}(f)} &
{ d^{\hskip.5pt 0*} (M_0) } \ar[d]^-{d^{1*}(f)} \cr
{ d^{\hskip.5pt 0*} (M_1) } \ar[r]^-{\epsilon_1} &
{ d^{1*} (M_1) } \cr
}$$
commute. In this case, we also say that the $A$-linear map
$f \colon M_0 \to M_1$ is horizontal with respect 
$\epsilon$. The category $\operatorname{Mod}_{(A,B)}$ of
$(A,B)$-modules admits a symmetric monoidal structure with the
monoidal product defined
by\index{symmetric monoidal product!of modules over cogroupoid}
$$(M_1,\epsilon_1) \otimes_{(A,B)} (M_2,\epsilon_2) = (M_1 \otimes_A
M_2, \epsilon_{12}),$$ where $\epsilon_{12}$ is the
unique map that makes the diagram
$$\xymatrix{
{ d^{1*}(M_1) \otimes_B d^{1*}(M_2) }\ar@{=}[r]
\ar[d]^-{\epsilon_1 \otimes \epsilon_2} &
{ d^{1*}(M_1 \otimes_A M_2) } 
\ar[d]^-{\epsilon_{12}} \cr
{ d^{\hskip.5pt 0*}(M_1) \otimes_B d^{\hskip.5pt 0*}(M_2) } \ar@{=}[r] &
{ d^{\hskip.5pt 0*}(M_1 \otimes_A M_2) } \cr
}$$
commute. The unit for the monoidal product is given by the
$A$-module $A$ with its unique structure of $(A,B)$-module, where
$\epsilon \colon d^{1*}(A) \to d^{\hskip.5pt 0*}(A)$ is the unique
$B$-linear map that makes the diagram
$$\begin{xy}
(0,5)*+{ s^{0*}d^{1*}(A) }="11";
(32,5)*+{ s^{0*}d^{\hskip.5pt 0*}(A) }="13";
(16,-5)*+{ A }="22";
{ \ar^-{s^{0*}(\epsilon)} "13";"11";};
{ \ar@{=} "22";"11";};
{ \ar@{=} "22";"13";};
\end{xy}$$
commute. \goodbreak

We again let $(A,B)$ be the cogroupoid associated with the
$\mathbb{E}_{\infty}$-ring $k$.\index{cogroupoid!associated with
$\mathbb{E}_{\infty}$-ring} For every spectrum $X$, we consider 
the cosimplicial spectrum $k^{\otimes [-]} \otimes X$. The homotopy
groups $\pi_*(k^{\otimes [0]} \otimes X)$ and
$\pi_*(k^{\otimes [1]} \otimes X)$ form a left $A$-module and a left
$B$-module, respectively. Moreover, we have $A$-linear maps
$$\xymatrix{
{ \pi_*(k^{\otimes [0]} \otimes X) } \ar[r]^-{d^i} &
{ d_*^i(\pi_*(k^{\otimes [1]} \otimes X)) } \cr
}$$
induced by $d^i \colon k^{\otimes [0]} \otimes X \to
k^{\otimes [1]} \otimes X$, and their adjunct maps
$$\xymatrix{
{ d^{i*}(\pi_*(k^{\otimes[0]} \otimes X)) } \ar[r]^-{\widetilde{d^i}} &
{ \pi_*(k^{\otimes [1]} \otimes X) } \cr
}$$
are $B$-linear isomorphisms. We now define the $(A,B)$-module
associated with the spectrum $X$ to be the pair $(M,\epsilon)$ with
\index{cogroupoid module!associated with spectrum}
$$M = \pi_*(k^{\otimes[0]} \otimes X)$$
and with $\epsilon$ the unique map that makes the following
diagram commute.
$$\begin{xy}
(0,6)*+{ d^{1*}(M) }="11";
(32,6)*+{ d^{\hskip.5pt 0*}(M) }="13";
(16,-6)*+{ \pi_*(k^{\otimes[1]} \otimes X) }="22";
{ \ar^-{\epsilon} "13";"11";};
{ \ar_-(.1){\widetilde{d^1}} "22";"11";};
{ \ar^-(.1){\widetilde{d^{\hskip.5pt 0}}} "22";"13";};
\end{xy}$$
We often abbreviate and write $\pi_*(k \otimes X)$ for the
$(A,B)$-module $(M,\epsilon)$. \goodbreak

The skeleton filtration of the cosimplicial spectrum
$k^{\otimes [-]} \otimes X$ gives rise to the conditionally convergent
$k$-based Adams spectral
sequence\index{Adams spectral sequence!$k$-based}
$$\textstyle{ E_{i,j}^2 = \Ext_{(\pi_*(k),\pi_*(k \otimes
  k))}^{-i}(\pi_*(k),\pi_*(k \otimes X))_j \Rightarrow
\pi_{i+j}(\lim_{\Delta} k^{\otimes[-]} \otimes X), }$$
where the $\Ext$-groups are calculated in the abelian category of
modules over the cogroupoid $(\pi_*(k),\pi_*(k \otimes k))$. An
$\mathbb{E}_{\infty}$-algebra structure on $X$ gives rise to a
commutative monoid structure on $\pi_*(k \otimes X)$ in the
symmetric monoidal category of
$(\pi_*(k),\pi_*(k \otimes k))$-modules and makes the spectral
sequence one of bigraded rings.

If $X$ is a $k$-module, then the augmented cosimplicial spectrum
$$\xymatrix{
{ X } \ar[r]^-{d^{\hskip.7pt 0}} &
{ k^{\otimes [-]} \otimes  X } \cr
}$$
acquires a nullhomotopy. Therefore, the spectral sequence collapses
and its edge homomorphism becomes an isomorphism
$$\xymatrix{
{ \pi_j(X) } \ar[r] &
{ \Hom_{(\pi_*(k),\pi_*(k \otimes k))}(\pi_*(k),\pi_*(k \otimes X))_j. } \cr
}$$
This identifies $\pi_j(X)$ with the subgroup of elements
$x \in \pi_j(k \otimes X)$ that are horizontal\footnote{\,In comodule
nomenclature, horizontal elements are called primitive elements.} with
respect to the stratification relative to
$(\pi_*(k),\pi_*(k \otimes k))$ in the sense that 
$\epsilon(1 \otimes_{A,d^1} x) = 1 \otimes_{A,d^{\hskip.7pt 0}} x$.
\index{horizontal!with respect to stratification}
\index{Adams spectral sequence|)}

\subsection{The B\"{o}kstedt spectral sequence}
\index{B\"{o}kstedt periodicity|(}
In general, if $R$ is an $\mathbb{E}_{\infty}$-ring, then
$$\THH(R) \simeq 
R^{\otimes \colim_{\Delta^{\op}} \Delta^1[-]/\partial\Delta^1[-]} \simeq
\colim_{\Delta^{\op}} R^{\otimes \Delta^1[-]/\partial\Delta^1[-]}.$$
A priori, the right-hand term is the colimit in
$\operatorname{Alg}_{\mathbb{E}_{\infty}}(\mathsf{Sp})$, but since the
index category $\Delta^{\op}$ is sifted~\cite[Lemma~5.5.8.4]{HTT},
the colimit agrees with the one in $\mathsf{Sp}$. The increasing
filtration of $\Delta^1[-]/\partial\Delta^1[-]$ by the skeleta induces
an increasing filtration of $\THH(R)$,
$$\Fil_0\THH(R) \to \Fil_1\THH(R) \to \cdots \to \Fil_n\THH(R) \to \cdots.$$
We let $k$ be an $\mathbb{E}_{\infty}$-ring and let $(A,B)$ be the
associated cogroupoid in graded rings, where $A = \pi_*(k)$ and
$B = \pi_*(k \otimes k)$. We also let $C$ be the commutative monoid
$\pi_*(k \otimes R)$ in the symmetric monoidal category of
$(A,B)$-modules. Here we assume that $d^{\hskip.7pt 0} \colon A \to B$
and $\eta \colon A \to C$ are flat. The filtration above gives rise to
a spectral sequence\index{B\"{o}kstedt spectral sequence}
$$E_{i,j}^2 = \HH_i(C/A)_j\Rightarrow
\pi_{i+j}(\hskip.5pt k \otimes \THH(R)),$$
called the B\"{o}kstedt spectral sequence. 
It is a spectral sequence of $C$-algebras in the symmetric monoidal
category of $(A,B)$-modules, and Connes' operator on
$\pi_*(k \otimes \THH(R))$ induces a map
$$\xymatrix{
{ E_{i,j}^r } \ar[r]^-{d} &
{ E_{i+1,j}^r } \cr
}$$
of spectral sequences, which, on the $E^2$-term, is equal to Connes' operator
$$\xymatrix{
{ \HH_i(C/A)_j } \ar[r]^-{d} &
{ \HH_{i+1}(C/A)_j. } \cr
}$$
In particular, if $y \in E_{i,j}^2$ and $dy \in E_{i+1,j}^2$ both
survive the spectral sequence, and if $y$ represents a homology class
$\tilde{y} \in \pi_{i+j}(k \otimes \THH(R))$, then $dy$ represents the
homology class $d\tilde{y} \in \pi_{i+j+1}(k \otimes \THH(R))$.

\begin{theorem}[B\"{o}kstedt]\label{thm:bokstedtperiodicity}The
canonical map of graded $\mathbb{F}_p$-algebras
$$\xymatrix{
{ \operatorname{Sym}_{\hskip.5pt\mathbb{F}_p}(\THH_2(\mathbb{F}_p)) } \ar[r] &
{ \THH_*(\mathbb{F}_p) } \cr
}$$
is an isomorphism, and $\THH_2(\mathbb{F}_p)$ is a $1$-dimensional
$\mathbb{F}_p$-vector space.
\end{theorem}

\begin{proof}We let $k = R = \mathbb{F}_p$ and continue to write
$A = \pi_*(k)$, $B = \pi_*(k \otimes k)$, and
$C = \pi_*(k \otimes R)$. We apply the
B\"{o}kstedt spectral sequence to show that, as a $C$-algebra in the
symmetric monoidal category of $(A,B)$-modules,
$$\pi_*(k \otimes \THH(R)) = C[x]$$
on a horizontal generator $x$ of degree $2$, and use the Adams
spectral sequence to conclude that $\THH_*(R) = R[x]$, as
desired. Along the way, we will use the fact, observed by
Angeltveit--Rognes~\cite{angeltveitrognes}, that the maps
$$\begin{xy}
(0,0)*+{ S^1 }="11";
(20,0)*+{ S^1 \vee S^1 }="12";
(40,0)*+{ \{\infty\} }="13";
(58,0)*+{ S^1 }="14";
(74,0)*+{ S^1 }="15";
(90,0)*+{ S^1 }="16";
{ \ar@<.7ex>^-{\psi} "12";"11";};
{ \ar@<.7ex>^-{\phi} "11";"12";};
{ \ar@<.7ex>^-{\eta} "14";"13";};
{ \ar@<.7ex>^-{\varepsilon} "13";"14";};
{ \ar^-{\chi} "16";"15";};
\end{xy}$$
where $\psi$ and $\phi$ are the pinch and fold maps, $\eta$ and
$\varepsilon$ the unique maps, and $\chi$ the flip map, give
$\pi_*(k \otimes \THH(R))$ the structure of a $C$-Hopf algebra
in the symmetric monoidal category of $(A,B)$-modules, assuming that the
unit map is flat. Moreover,
the B\"{o}kstedt spectral sequence is a spectral sequence of $C$-Hopf
algebras, provided that the unit map $C \to E^r$ is flat for all
$r \geq 2$. We remark that the requirement that the comultiplication on
$E = E^r$ be a map of $C$-modules in the symmetric monoidal category
of $(A,B)$-modules is equivalent to the requirement that the diagram
$$\begin{xy}
(0,7)*+{ d^{1*}(E) }="11";
(30,7)*+{ d^{1*}(E \otimes_CE) }="12";
(70,7)*+{ d^{1*}(E) \otimes_{d^{1*}(C)}d^{1*}(E) }="13";
(0,-7)*+{ d^{\hskip.7pt 0*}(E) }="21";
(30,-7)*+{ d^{\hskip.7pt 0*}(E \otimes_CE) }="22";
(70,-7)*+{ d^{\hskip.7pt 0*}(E) \otimes_{d^{\hskip.7pt
      0*}(C)}d^{\hskip.7pt 0*}(E) }="23";
{ \ar^-{d^{1*}(\psi)} "12";"11";};
{ \ar@{=} "13";"12";};
{ \ar^-{d^{0*}(\psi)} "22";"21";};
{ \ar@{=} "23";"22";}; 
{ \ar^-{\epsilon_E} "21";"11";};
{ \ar^-{\epsilon_E \otimes_{\epsilon_C} \epsilon_E} "23";"13";};
\end{xy}$$
commutes. 

To begin, we recall from Milnor~\cite{milnor} that, as a graded
$A$-algebra,
$$C = \begin{cases}
A[\,\bar{\xi}_i \mid i \geq 1\, ] & \text{for $p = 2$,} \cr
A[\,\bar{\xi}_i \mid i \geq 1\, ] \otimes_A
\Lambda_A\{\bar{\tau}_i \mid i \geq 0\} & \text{for $p$ odd,} \cr
\end{cases}$$
where $\bar{\tau}_i = \chi(\tau_i)$ and $\bar{\xi}_i = \chi(\xi_i)$
are the images by the antipode of Milnor's generators $\tau_i$ and
$\xi_i$. Here, for $p = 2$, $\deg(\bar{\xi}_i) = p^i-1$, while for $p$
odd, $\deg(\bar{\xi}_i) = 2(p^i - 1)$ and $\deg(\bar{\tau}_i) =
2p^i-1$. The stratification
$$\xymatrix{
{ d^{1*}(C) } \ar[r]^-{\epsilon} &
{ d^{\hskip.7pt 0*}(C) } \cr
}$$
is given by
$$\begin{aligned}
{} \epsilon(1 \otimes_{A,d^1} \bar{\xi}_i)
& = \textstyle{\sum}\, \bar{\xi}_s \otimes_{A,d^{\hskip.7pt 0}}
\bar{\xi}_t^{\,p^s} \cr
{} \epsilon(1 \otimes_{A,d^1} \bar{\tau}_i)
& = 1 \otimes_{A,d^{\hskip.7pt 0}} \bar{\tau}_i + \textstyle{\sum}\,
\bar{\tau}_s \otimes_{A,d^{\hskip.7pt 0}} \bar{\xi}_t^{\,p^s} \cr
\end{aligned}$$
with the sums indexed by $s,t \geq 0$ with $s+t = i$. Moreover, we
recall from Steinberger~\cite{steinberger} that the power operations
on $C = \pi_*(k \otimes R)$ satisfy
$$\begin{aligned}
Q^{p^i}(\bar{\xi}_i) & = \bar{\xi}_{i+1}, \cr
Q^{p^i}(\bar{\tau}_i) & = \bar{\tau}_{i+1}. \cr
\end{aligned}$$
A very nice brief account of this calculation is given in~\cite{wang}.

We first consider $p = 2$. The $E^2$-term of the B\"{o}kstedt spectral
sequence, as a $C$-Hopf algebra in $(A,B)$-modules, takes the form
$$E^2 = \Lambda_C\{d\bar{\xi}_i \mid i \geq 1\}$$
with $\deg(\bar{\xi}_i) = (0,2^i-1)$ and
$\deg(d\bar{\xi}_i) = (1,2^i-1)$, and all differentials in the
spectral sequence vanish. Indeed, they are $C$-linear derivations and,
for degree reasons, the algebra generators $d\bar{\xi}_i$ cannot
support non-zero differentials. We define $x$ to be the image of
$\bar{\xi}_1$ by the composite map
$$\xymatrix{
{ \pi_1(k \otimes R) } \ar[r]^-{\eta} &
{ \pi_1(k \otimes \THH(R)) } \ar[r]^-{d} &
{ \pi_2(k \otimes \THH(R)), } \cr
}$$
and proceed to show, by induction on $i \geq 0$, that the homology
class $x^{2^i}$ is represented by the element $d\bar{\xi}_{i+1}$ in
the spectral sequence. The case $i = 0$ follows from what was said
above, so we assume the statement has been proved for $i = r-1$ and
prove it for $i=r$. We have
$$x^{2^r} = (x^{2^{r-1}})^2 = Q^{2^r}(x^{2^{r-1}}),$$
which, by induction, is represented by
$Q^{2^r}(d\bar{\xi}_r)$. But Proposition~\ref{prop:connesdyerlashof}
and Steinberger's calculation show that
$$Q^{2^r}(d\bar{\xi}_r) = d(Q^{2^r}(\bar{\xi}_r)) = d\bar{\xi}_{r+1},$$
so we conclude that $x^{2^r}$ is represented by $d\bar{\xi}_{r+1}$.
Hence, as a graded $C$-algebra, 
$$\pi_*(k \otimes \THH(R)) = C[x].$$
Finally, we calculate that
$$\begin{aligned}
{} \epsilon(1 \otimes_{A,d^1} x) 
{} & = \epsilon((\id \otimes_{A,d^1} \, d)(\eta \otimes_{A,d^1}
\eta)(1 \otimes_{A,d^1} \bar{\xi}_1)) \cr
{} & = (\id \otimes_{A,d^{\hskip.7pt 0}}\, d)(\eta \otimes_{A,d^0}
\eta)(\epsilon (1 \otimes_{A,d^1}
\bar{\xi}_1)) \cr
{} & = (\id \otimes_{A,d^{\hskip.7pt 0}}\, d)(\eta \otimes_{A,d^0}
\eta) (\bar{\xi}_1 \otimes_{A,d^{\hskip.7pt 0}} 1 + 1 \otimes_{A,d^{\hskip.7pt 0}}
\bar{\xi}_1) \cr
{} & = 1 \otimes_{A,d^{\hskip.7pt 0}} x, \cr
\end{aligned}$$ 
which shows that the element $x$ is horizontal with respect to the
stratification of $\pi_*(k \otimes \THH(R))$ relative to $(A,B)$. 

We next let $p$ be odd. As a $C$-Hopf algebra,
$$E^2 = \Lambda_C\{d\bar{\xi}_i \mid i \geq 1\} \otimes_C
\Gamma_C\{d\bar{\tau}_i \mid i \geq 0 \}$$
with $\deg(\bar{\xi}_i) = (0,2p^i-2)$, $\deg(d\bar{\xi}_i) =
(1,2p^i-2)$, $\deg(\bar{\tau}_i) = (0,2p^i-1)$, and
$\deg(d\bar{\tau}_i) = (1,2p^i-1)$, and with the coproduct given by
$$\begin{aligned}
\psi(d\bar{\xi}_i) & = 1 \otimes d\bar{\xi}_i + d\bar{\xi}_i \otimes 1, \cr
\psi((d\bar{\tau}_i)^{[r]}) & = \textstyle{\sum}\, (d\bar{\tau}_i)^{[s]} \otimes
(d\bar{\tau}_i)^{[t]}, \cr
\end{aligned}$$
where the sum ranges over $s,t \geq 0$ with $s+t=r$. Here $(-)^{[r]}$
indicates the $r$th divided power. We define $x$ to be the image of
$\bar{\tau}_0$ by the composite map
$$\xymatrix{
{ \pi_1(k \otimes R) } \ar[r]^-{\eta} &
{ \pi_1(k \otimes \THH(R)) } \ar[r]^-{d} &
{ \pi_2(k \otimes \THH(R)) } \cr
}$$
and see, as in the case $p = 2$, that the homology class $x^{p^i}$ is
represented by the element $d\bar{\tau}_i$. This is also shows that
for $i \geq 1$, the element
$$d\bar{\xi}_i = d(\beta(\bar{\tau}_i)) = \beta(d\bar{\tau}_i)$$
represents the homology class $\beta(x^{p^i}) = p^ix^{p^i-1}\beta(x)$,
which is zero. Hence, this element is annihilated by some
differential. We claim that for all $i,s \geq 0$,
$$d^{p-1}((d\bar{\tau}_i)^{[p+s]}) = a_i \cdot d\bar{\xi}_{i+1} \cdot
(d\bar{\tau}_i)^{[s]}$$
with $a_i \in k^{*}$ a unit that depends on $i$ but not on
$s$. Grating this, we find as in the case $p = 2$ that, as a
$C$-algebra,
$$\pi_*(k \otimes \THH(R)) = C[x]$$
with $x$ horizontal of degree $2$, which proves the theorem. 

To prove the claim, we note that a shortest possible non-zero
differential between elements of lowest possible total degree factors
as a composition
$$\begin{xy}
(0,0)*+{ E_{i,j}^r }="1";
(20,0)*+{ QE_{i,j}^r }="2";
(47,0)*+{ PE_{i-r,j+r-1}^r }="3";
(78,0)*+{ E_{i-r,j+r-1}^r, }="4";
{ \ar@<.2ex>^-{\pi} "2";"1";};
{ \ar@<.2ex>^-{\bar{d}^r} "3";"2";};
{ \ar@<.2ex>^{\iota} "4";"3";};
\end{xy}$$
where $\pi$ is the quotient by algebra decomposables and $\iota$ is the
inclusion of the coalgebra primitives. We further observe that
$\smash{ QE_{i,j}^2 }$ is zero, unless $i$ is a power of $p$, and that 
$\smash{ PE_{i,j}^2 }$ is zero, unless $i = 1$. Hence, the shortest
possible non-zero differential of lowest possible total degree is
$$d^{p-1}((d\bar{\tau}_0)^{[p]}) = a_0 \cdot d\bar{\xi}_1$$
with $a_0 \in k$. In particular, we have $E^{p-1} = E^2$. If
$a_0 = 0$, then $d\bar{\xi}_1$ survives the spectral sequence, so 
$a_0 \in k^*$. This proves the claim for $i = s = 0$.

We proceed by nested induction on $i,s \geq 0$ to prove the claim
in general. We first note that if, for a fixed $i \geq 0$, the claim
holds for $s = 0$, then it holds for all $s \geq 0$. For let $s \geq 1$
and assume, inductively, that the claim holds for all smaller values
of $s$. One calculates that the difference
$$d^{p-1}((d\bar{\tau}_i)^{[p+s]}) - a_i \cdot d\bar{\xi}_{i+1} \cdot
(d\bar{\tau}_i)^{[s]}$$
is a coalgebra primitive element, which shows that it is zero, since
all non-zero coalgebra primitives in $E^{p-1} = E^2$ have filtration
$i = 1$. 

It remains to prove that the claim holds for all $i \geq 0$ and
$s = 0$. We have already proved the case $i = 0$, so we let $j \geq 1$
and assume, inductively, that the claim has been proved for all
$i < j$ and all $s \geq 0$. The inductive assumption implies that
$E^p$ is a subquotient of the $C$-subalgebra 
$$D = \Lambda_C\{d\bar{\xi}_i \mid i \geq j+1\} \otimes_C
\Gamma_C\{d\bar{\tau}_i \mid i \geq j\} \otimes_C
C[d\bar{\tau}_i \mid i < j ] / ((d\bar{\tau}_i)^p \mid i < j)$$
of $E^{p-1} = E^2$. Now, since $\pi_*(k \otimes \THH(R))$ is an
augmented $C$-algebra, all elements of filtration $0$ survive the
spectral sequence. Hence, if $x \in E^r$ with $r \geq p$ supports a
non-zero differential, then $x$ has filtration at least $p+1$. But all
algebra generators in $D$ of filtration at least $p+1$ have total
degree at least $2p^{j+2}$, so either
$d^{p-1}((d\bar{\tau}_j)^{[p]})$ is non-zero, or else all elements in 
$D$ of total degree at most $2p^{j+2}-2$ survive the spectral
sequence. Since we know that the element $d\bar{\xi}_{j+1} \in D$ of
total degree $2p^{j+1}-1$ does not survive the spectral sequence, we
conclude that the former is the case. We must show that
$$d^{p-1}((d\bar{\tau}_j)^{[p]}) = a_j \cdot d\bar{\xi}_{j+1}$$
with $a_j \in k^*$, and to this end, we use the fact that $E^{p-1}$ is a
$C$-Hopf algebra in the symmetric monoidal category
of $(A,B)$-modules. We have
$$\begin{aligned}
\epsilon(1 \otimes_{A,d^1}d\bar{\xi}_i)
& = \epsilon((\id \otimes_{A,d^1}\,d)(1 \otimes_{A,d^1} \bar{\xi}_i))
{} = (\id \otimes_{A,d^{\hskip.7pt 0}}\,d)(\epsilon(1 \otimes_{A,d^1}
\bar{\xi}_i)) \cr
{} & = (\id \otimes_{A,d^{\hskip.7pt 0}}\,d)(\textstyle{\sum}\,
\bar{\xi}_s \otimes_{A,d^{\hskip.7pt 0}} \bar{\xi}_t^{\,p^s}) 
{} = 1 \otimes_{A,d^{\hskip.7pt 0}} d\bar{\xi}_i, \cr
\epsilon(1 \otimes_{A,d^1}d\bar{\tau}_i)
& = \epsilon((\id \otimes_{A,d^1}\,d)(1 \otimes_{A,d^1} \bar{\tau}_i))
{} = (\id \otimes_{A,d^{\hskip.7pt 0}}\,d)(\epsilon(1 \otimes_{A,d^1}
\bar{\tau}_i)) \cr
{} & = (\id \otimes_{A,d^{\hskip.7pt 0}}\,d)(1
\otimes_{A,d^{\hskip.7pt 0}}\bar{\tau}_i + \textstyle{\sum}\,
\bar{\tau}_s \otimes_{A,d^{\hskip.7pt 0}} \bar{\xi}_t^{\,p^s}) 
{} = 1 \otimes_{A,d^{\hskip.7 pt 0}} d\bar{\tau}_i, \cr
\end{aligned}$$
where the sums range over $s,t \geq 0$ with $s+t = i$.
Hence, the sub-$k$-vector space of $E^{p-1}$ that consists of
the horizontal elements of bidegree $(1,2p^{j+1}-2)$ is spanned by 
$d\bar{\xi}_{j+1}$. Therefore, it suffices to show that
$(d\bar{\tau}_j)^{[p]}$, and hence, $d^{p-1}((d\bar{\tau}_j)^{[p]})$
is horizontal. We have already proved that $d\bar{\tau}_j$ is
horizontal, and using the fact that $E^{p-1}$ is a $C$-algebra in
the symmetric monoidal category of $(A,B)$-modules, we conclude that
$(d\bar{\tau}_j)^s$, and therefore, $(d\bar{\tau}_j)^{[s]}$ is
horizontal for all $0 \leq s < p$. Finally, we make use of the fact
that $E^{p-1}$ is a $C$-coalgebra in the symmetric monoidal category
of $(A,B)$-modules. Since 
$$\psi((d\bar{\tau}_j)^{[p]})
= \textstyle{\sum}\, (d\bar{\tau}_j)^{[s]} \otimes (d\bar{\tau}_j)^{[t]}$$
with the sum indexed by $s,t \geq 0$ with $s+t=p$, and since we have
already proved that $(d\bar{\tau}_j)^{[s]}$ with $0 \leq s < p$ are
horizontal, we find that $(d\bar{\tau}_j)^{[p]}$ is horizontal. This
completes the proof.
\end{proof}

We finally recall the following analogue of the Segal conjecture. This
is a key result for understanding topological cyclic homology and its
variants.

\begin{addendum}\label{add:bokstedtperiodicity}The Frobenius induces
an equivalence
$$\xymatrix{
{ \THH(\mathbb{F}_p) } \ar[r]^-{\varphi} &
{ \tau_{\geq 0}\THH(\mathbb{F}_p)^{tC_p}. } \cr
}$$
\end{addendum}

\begin{proof}See~\cite[Section~IV.4]{nikolausscholze}.
\end{proof}

\index{B\"{o}kstedt periodicity|)}

\section{Perfectoid rings}\label{sec:perfectoidrings}

Perfectoid rings are to topological Hochschild homology what separably
closed fields are to $K$-theory: they annihilate K\"{a}hler differentials.
In this section, we present the proof by
Bhatt--Morrow--Scholze that
B\"{o}kstedt periodicity holds for every perfectoid ring $R$. As a
consequence, the Tate spectral sequence
$$E_{i,j}^2 = \hat{H}^{-i}(B\hskip.5pt\mathbb{T},\THH_j(R,\mathbb{Z}_p))
\Rightarrow \TP_{i+j}(R,\mathbb{Z}_p)$$
collapses and gives the ring $\TP_0(R,\mathbb{Z}_p)$ a complete and
separated descending filtration, the graded pieces of which are free
$R$-modules of rank $1$. The ring $\TP_0(R,\mathbb{Z}_p)$, however, is
not a power series ring over $R$. Instead, it agrees, up to unique
isomorphism over $R$, with Fontaine's $p$-adic period ring
$A_{\inf}(R)$, the definition of which we recall below. Finally, we
use these results to prove that B\"{o}kstedt periodicity implies Bott
periodicity.

\subsection{Perfectoid rings}
\index{perfectoid ring|(}

A $\mathbb{Z}_p$-algebra $R$ is perfectoid, for example, if there exists a
non-zero-divisor $\pi \in R$ with $p \in \pi^pR$ such that $R$ is
complete and separated with respect to the $\pi$-adic topology and
such that the Frobenius $\varphi \colon R/\pi \to R/\pi^p$ is an
isomorphism. We will give the general definition, which does
not require $\pi \in R$ to be a non-zero-divisor, below.
Typically, perfectoid rings are large and highly non-noetherian.
Moreover, the ring $R/\pi$ is typically not a field,
but is also a large non-noetherian ring with many nilpotent
elements. An example to keep in mind is the valuation ring
$\mathcal{O}_C$ in an algebraically closed field $C/\mathbb{Q}_p$
that is complete with respect to a non-archimedean absolute value
extending the $p$-adic absolute value on $\mathbb{Q}_p$; here we can
take $\pi$ to be a $p$th root of $p$.

We recall some facts from~\cite[Section~3]{bhattmorrowscholze1}. If a
ring $S$ contains an element $\pi \in S$ such that $p \in \pi S$ and
such that $\pi$-adic topology on $S$ is complete and separated, then
the canonical projections
$$\begin{xy}
(0,7)*+{ \lim_{n,F} W(S) }="11";
(33,7)*+{ \lim_{n,F} W(S/p) }="12";
(68,7)*+{ \lim_{n,F} W(S/\pi) }="13";
(0,-7)*+{ \lim_{n,F} W_n(S) }="21";
(33,-7)*+{ \lim_{n,F} W_n(S/p) }="22";
(68,-7)*+{ \lim_{n,F} W_n(S/\pi) }="23";
{ \ar "12";"11";};
{ \ar "13";"12";};
{ \ar "21";"11";};
{ \ar "22";"12";};
{ \ar "23";"13";};
{ \ar "22";"21";};
{ \ar "23";"22";};
\end{xy}$$
all are isomorphisms. Here the limits range over non-negative integers
$n$ with the respective Witt vector Frobenius maps as the structure
maps. Moreover, since the Witt vector Frobenius for $\mathbb{F}_p$-algebras
agrees with the map of rings of Witt vectors induced by the Frobenius,
we have a canonical map
$$\begin{xy}
(0,0)*+{ W(\lim_{n,\varphi} S/p) }="1";
(35,0)*+{ \lim_{n,F} W(S/p), }="2";
{ \ar "2";"1";};
\end{xy}$$
and this map, too, is an isomorphism, since the Witt vector functor
preserves limits. The perfect $\mathbb{F}_p$-algebra
$$S^{\flat} = \lim_{n,\varphi} S/p = \lim(S/p \xleftarrow{\varphi} S/p
\xleftarrow{\varphi} \cdots)$$
is called the tilt\index{tilt} of $S$, and its ring of Witt vectors
$$A_{\inf}(S) = W(S^{\flat})$$
is called Fontaine's ring of $p$-adic
periods.\index{Fontaine's ring of $p$-adic periods} The Frobenius
automorphism $\varphi$ of $S^{\flat}$ induces the automorphism
$W(\varphi)$ of $A_{\inf}(S)$, which, by abuse of notation, we also
write $\varphi$ and call the Frobenius.

We again consider the diagram of isomorphisms at the beginning of the
section. By composing the isomorphisms in the diagram with the projection
onto $W_n(S)$ in the lower left-hand term of the diagram, we obtain a ring
homomorphism $\smash{ \widetilde{\theta}_n \colon A_{\inf}(S) \to
  W_n(S) }$, and we define
$$\begin{xy}
(0,0)*+{ A_{\inf}(S) }="1";
(25,0)*+{ W_n(S) }="2";
{ \ar^-{\theta_n} "2";"1";};
\end{xy}$$
to be $\theta_n = \widetilde{\theta}_n \circ \varphi^n$. It is clear
from the definition that the diagrams
$$\begin{xy}
(0,7)*+{ A_{\inf}(S) }="11";
(27,7)*+{ W_n(S) }="12";
(57,7)*+{ A_{\inf}(S) }="13";
(84,7)*+{ W_n(S) }="14";
(0,-7)*+{ A_{\inf}(S) }="21";
(27,-7)*+{ W_{n-1}(S) }="22";
(57,-7)*+{ A_{\inf}(S) }="23";
(84,-7)*+{ W_{n-1}(S) }="24";
{ \ar^-{\theta_n} "12";"11";};
{ \ar^-{\widetilde{\theta}_n} "14";"13";};
{ \ar@<-.8ex>_-{\id} "21";"11";};
{ \ar@<.8ex>^-{\varphi} "21";"11";};
{ \ar@<-.8ex>_-{R} "22";"12";};
{ \ar@<.8ex>^-{F} "22";"12";};
{ \ar@<-.8ex>_-{\varphi^{-1}} "23";"13";};
{ \ar@<.8ex>^-{\id} "23";"13";};
{ \ar@<-.8ex>_-{R} "24";"14";};
{ \ar@<.8ex>^-{F} "24";"14";};
{ \ar^-{\theta_{n-1}} "22";"21";}; 
{ \ar^-{\widetilde{\theta}_{n-1}} "24";"23";};
\end{xy}$$
commute. The map $\theta = \theta_1 \colon A_{\inf}(S) \to S = W_1(S)$
is Fontaine's map from~\cite{fontaine}, which we now describe more
explicitly. There is a well-defined map
$$\begin{xy}
(0,0)*+{ {}_{\phantom{\flat}}S^{\flat} }="1";
(20,0)*+{ S }="2";
{ \ar^-{(-)^{\#}} "2";"1";};
\end{xy}$$
that to $a = (x_0,x_1, \dots) \in \lim_{n,\varphi} S/p = S^{\flat}$
assigns $a^{\#} = \lim_{n \to \infty} y_n^{p^n}$ for any choice of lifts
$y_n \in S$ of $x_n \in S/p$. It is multiplicative, but it is not additive
unless $S$ is an $\mathbb{F}_p$-algebra. Using this map, we have
$$\textstyle{ \theta( \sum_{i \geq 0}[a_i]\,p^i ) = \sum_{i \geq
    0}a_i^{\#}p^i, }$$
where $[-] \colon S^{\flat} \to W(S^{\flat})$ is the Teichm\"{u}ller
representative. We can now state the general definition of a
perfectoid ring that is used in~\cite{bhattmorrowscholze1},
\cite{bhattmorrowscholze2}, and~\cite{bhattscholze}.
\index{perfectoid ring!definition of}

\begin{definition}\label{def:perfectoid}A $\mathbb{Z}_p$-algebra $R$
is perfectoid if there exists $\pi \in R$ such that $p \in \pi^pR$,
such that the $\pi$-adic topology on $R$ is complete and separated,
such that the Frobenius $\varphi \colon R/p \to R/p$ is
surjective, and such that the kernel of 
$\theta \colon A_{\inf}(R) \to R$ is a principal ideal.
\end{definition}

The ideal $I = \ker(\theta) \subset A = W(R^{\flat})$ is
typically not fixed by the Frobenius on $A$, but it always satisfies
the prism property that\index{prism}\index{untilt}
$$p \in I + \varphi(I)A.$$
If an ideal $J \subset A$ satisfies the prism property, then the
quotient $A/J$ is an untilt of $R^{\flat}$ in the sense that it is
perfectoid and that its tilt is $R^{\flat}$. In fact, every untilt of
$\smash{ R^{\flat} }$ arises as $A/J$ for some ideal $J \subset A$
that satisfies the prism property. The set of such ideals is typically
large, but it has a very interesting $p$-adic geometry. Indeed, for
$R = \mathcal{O}_C$, there is a canonical  one-to-one correspondence
between orbits under the Frobenius of such ideals and closed points of
the Fargues--Fontaine curve\index{Fargues--Fontaine curve}
$\operatorname{FF}_C$~\cite{fargues}. Among all ideals $J \subset A$
that satisfy the prism property, the ideal $J = pA$ is the only one
for which the untilt $A/J = \smash{ R^{\flat} }$ is of characteristic
$p$; all other untilts $A/J$ are of mixed characteristic $(0,p)$. One
can show that every untilt $A/J$ is a reduced ring and that a
generator $\xi$ of the ideal $J \subset A$ necessarily is a
non-zero-divisor. Hence, such a generator is well-defined, up to a
unit in $A$. An untilt $A/J$ may have $p$-torsion, but if an element
is annihilated by some power of $p$, then it is in fact annihilated by
$p$. We refer to~\cite{bhattscholze} for proofs of these statements.
\index{perfectoid ring!properties of}

Bhatt--Morrow--Scholze prove
in~\cite[Theorem~6.1]{bhattmorrowscholze2} that 
B\"{o}kstedt periodicity for $R = \mathbb{F}_p$ implies the analogous
result for $R$ any perfectoid ring.
\index{B\"{o}kstedt periodicity!for perfectoid ring}

\begin{theorem}[Bhatt--Morrow--Scholze]\label{thm:bokstedtperiodicityperfectoid}If
$R$ is a perfectoid ring, then the canonical map is an isomorphism of graded rings
$$\xymatrix{
{ \operatorname{Sym}_R(\THH_2(R,\mathbb{Z}_p)) } \ar[r] &
{ \THH_*(R,\mathbb{Z}_p), } \cr
}$$
and $\THH_2(R,\mathbb{Z}_p)$ is a free $R$-module of rank $1$.
\end{theorem}

\begin{proof}We follow the proof in loc.~cit. We first claim that the
canonical map
$$\xymatrix{
{ \Gamma_R(\HH_2(R/\mathbb{Z},\mathbb{Z}_p)) } \ar[r] &
{ \HH_*(R/\mathbb{Z},\mathbb{Z}_p) } \cr
}$$
is an isomorphism and that the $R$-module
$\HH_2(R/\mathbb{Z},\mathbb{Z}_p)$ is free of rank $1$. To prove this,
we first notice that the base-change map
$$\xymatrix{
{ \HH(R/\mathbb{Z}) } \ar[r] &
{ \HH(R/A_{\inf}(R)) } \cr
}$$
is a $p$-adic equivalence. Indeed, we always have
$$\HH(R/A_{\inf}(R)) \simeq \HH(R/\mathbb{Z}_p)
\otimes_{\HH(A_{\inf}(R)/\mathbb{Z}_p)} A_{\inf}(R),$$
and $\HH(A_{\inf}(R)/\mathbb{Z}) \to A_{\inf}(R)$ is a $p$-adic
equivalence, because
$$\HH(A_{\inf}(R)/\mathbb{Z}) \otimes_{\mathbb{Z}} \mathbb{F}_p
\simeq \HH(A_{\inf}(R) \otimes_{\mathbb{Z}} \mathbb{F}_p/\mathbb{F}_p)
\simeq \HH(R^{\flat} / \mathbb{F}_p) \simeq R^{\flat}.$$
The last equivalence holds since $R^{\flat}$ is perfect.
Now, we write $R \simeq \Lambda_{A_{\inf}(R)}\{y\}$ with $dy = \xi$ to
see that $R \otimes_{A_{\inf}(R)} R \simeq \Lambda_R\{y\}$ with $dy = 0$,
and similarly, we write $R \simeq \Lambda_R\{y\} \otimes
\Gamma_R\{x\}$ with $dx^{[i]} = x^{[i-1]}y$ to see that
$$\HH(R/A_{\inf}(R)) \simeq R \otimes_{R \otimes_{A_{\inf}(R)} R} R
\simeq \Gamma_R\{x\},$$
which proves the claim.

It follows, in particular, that the $R$-module
$$\THH(R,\mathbb{Z}_p) \otimes_{\THH(\mathbb{Z})}
\mathbb{Z} \simeq \HH(R/\mathbb{Z},\mathbb{Z}_p)$$
is pseudocoherent\index{pseudocoherent module} in the sense that it
can be represented by a chain complex of finitely generated free
$R$-modules that is bounded below. Since $\THH(\mathbb{Z})$ has
finitely generated homotopy groups, we conclude, inductively, that
$$\THH(R,\mathbb{Z}_p) \otimes_{\THH(\mathbb{Z})}\tau_{\leq
  n}\THH(\mathbb{Z})$$
is a pseudocoherent $R$-module for all $n \geq 0$. Therefore, also
$\THH(R,\mathbb{Z}_p)$ is a pseudocoherent $R$-module.

Next, we claim that any ring homomorphism $R \to R'$ between
perfectoid rings induces an equivalence
$$\xymatrix{
{ \THH(R,\mathbb{Z}_p) \otimes_RR' } \ar[r] &
{ \THH(R',\mathbb{Z}_p). } \cr
}$$
Indeed, it suffices to prove that the claim holds after extension of
scalars along the canonical map $\THH(\mathbb{Z}) \to \mathbb{Z}$.
This reduces us to proving that
$$\xymatrix{
{ \HH(R,\mathbb{Z}_p) \otimes_RR' } \ar[r] &
{ \HH(R',\mathbb{Z}_p) } \cr
}$$
is an equivalence, which follows from the first claim.

We now prove that the map in the statement is an isomorphism. The
case of $R = \mathbb{F}_p$ is Theorem~\ref{thm:bokstedtperiodicity},
and the case of a perfect $\mathbb{F}_p$-algebra follows from the
base-change formula that we just proved. In the general case, we show,
inductively, that the map is an isomorphism in degree $i \geq 0$. So
we assume that the map is an isomorphism in degrees $< i$ and prove
that it is an isomorphism in degree $i$. By induction, the $R$-module
$\tau_{<i}\THH(R,\mathbb{Z}_p)$ is perfect, and hence, the $R$-module
$\tau_{\geq i}\THH(R,\mathbb{Z}_p)$ is pseudocoherent. It follows that
the $R$-module $\THH_i(R,\mathbb{Z}_p)$ is finitely generated. Since
$R$ is perfectoid, the composition 
$$\xymatrix{
{ R } \ar[r] &
{ R/p } \ar[r] &
{ \bar{R} = \colim_{n,\varphi} R/p = R/\sqrt{pR} } \cr
}$$
of the canonical projection and the canonical map from the initial term
in the colimit is surjective. Since $\bar{R}$
is a perfect $\mathbb{F}_p$-algebra, the base-change formula and the
inductive hypothesis show that, in the diagram
$$\xymatrix{
{ \operatorname{Sym}_R(\THH_2(R,\mathbb{Z}_p)) \otimes_R \bar{R} } \ar[r] \ar[d] &
{ \THH_*(R, \mathbb{Z}_p) \otimes_R \bar{R} } \ar[d] \cr
{ \operatorname{Sym}_{\bar{R}}(\THH_2(\bar{R},\mathbb{Z}_p)) } \ar[r] &
{ \THH_*(\bar{R}, \mathbb{Z}_p), } \cr
}$$
the vertical maps are isomorphisms in degrees $\leq i$, and we have
already seen that the lower horizontal map is an isomorphism. Hence,
the upper horizontal map is an isomorphism in degrees $\leq i$. Since
the kernel of the map $R \to \bar{R}$ contains the Jacobson radical,
Nakayama's lemma shows that the map in the statement of the theorem
is surjective in degrees $\leq i$. To prove that it is also injective,
we consider the diagram
$$\xymatrix{
{ \operatorname{Sym}_R(\THH_2(R,\mathbb{Z}_p)) } \ar[r] \ar[d] &
{ \THH_*(R,\mathbb{Z}_p) } \ar[d] \cr
{ \prod_{\mathfrak{p}} \operatorname{Sym}_R(\THH_2(R,\mathbb{Z}_p)) \otimes_R
  R_{\mathfrak{p}} } \ar[r] &
{ \prod_{\mathfrak{p}} \THH_*(R,\mathbb{Z}_p) \otimes_R R_{\mathfrak{p}} } \cr
}$$
where the products range over the minimal primes $\mathfrak{p} \subset
R$. Since $R$ is reduced, the left-hand vertical map is injective and 
the local rings $R_{\mathfrak{p}}$ are fields. Hence, it suffices to
prove that for every minimal prime $\mathfrak{p} \subset R$, the
map
$$\xymatrix{
{ \operatorname{Sym}_R(\THH_2(R,\mathbb{Z}_p)) \otimes_R
  R_{\mathfrak{p}} } \ar[r] &
{ \THH_*(R,\mathbb{Z}_p) \otimes_R R_{\mathfrak{p}} } \cr
}$$
is injective in degrees $\leq i$. To this end, we write
$\operatorname{Spec}(R)$ as the union of the closed subscheme
$\operatorname{Spec}(\bar{R})$ and its open complement
$\operatorname{Spec}(R[\nicefrac{1}{p}])$. If $\mathfrak{p}$ belongs to
$\operatorname{Spec}(\bar{R})$, then the map in question is an
isomorphism in degrees $\leq i$ by what was proved above. Similarly,
if $\mathfrak{p}$ belongs to $\operatorname{Spec}(R[\nicefrac{1}{p}])$,
then the map in a question is an isomorphism in all degrees, since the
map
$$\xymatrix{
{ \operatorname{Sym}_R(\THH_2(R,\mathbb{Z}_p)) \otimes_R R[\nicefrac{1}{p}] } \ar[r] &
{ \THH_*(R,\mathbb{Z}_p) \otimes_R R[\nicefrac{1}{p}] } \cr
}$$
is so, by the claim at the beginning of the proof. This completes the
proof.
\end{proof}

\begin{addendum}\label{add::bokstedtperiodicityperfectoid}If $R$ is a
perfectoid ring, then
$$\xymatrix{
{ \THH(R,\mathbb{Z}_p) } \ar[r]^-{\varphi} &
{ \tau_{\geq 0}\THH(R,\mathbb{Z}_p)^{tC_p} } \cr
}$$
is an equivalence.
\end{addendum}

\begin{proof}See~\cite[Proposition~6.2]{bhattmorrowscholze2}.
\end{proof}

We show that Fontaine's map $\theta \colon A_{\inf}(R) \to R$ is
the universal $p$-complete pro-infinitesimal thickening
following~\cite[Th\'{e}or\`{e}me~1.2.1]{fontaine}. We remark that, in
loc.~cit., Fontaine defines $A_{\inf}(R)$ to be the
$\ker(\theta)$-adic completion of $W(\smash{ R^{\flat} })$. We include
a proof here that this is not necessary in that the
$\ker(\theta)$-adic topology on $W(\smash{ R^{\flat} })$ is already
complete and separated.
\index{$p$-complete pro-infinitesimal thickening}
\index{Fontaine's ring of $p$-adic periods}

\begin{proposition}[Fontaine]\label{prop:fontaine}If $R$ is a perfectoid
ring, then the map
$$\xymatrix{
{ A_{\inf}(R) } \ar[r]^-{\theta} &
{ R } \cr
}$$
is initial among ring
homomorphisms $\theta_D \colon D \to R$ such that $D$ is complete and
separated in both the $p$-adic topology and the $\ker(\theta_D)$-adic
topology.
\end{proposition}

\begin{proof}We first show that $A = A_{\inf}(R)$ is complete and
separated in both the $p$-adic topology and the $\ker(\theta)$-adic
topology. Since $R^{\flat}$ is a perfect $\mathbb{F}_p$-algebra, we
have $p^nA = V^n(A) \subset A$, so the $p$-adic topology on
$A = W(R^{\flat})$ is complete and separated. Moreover, since 
$p \in A$ is a non-zero-divisor, this is equivalent to $A$ being
derived $p$-complete. As we recalled above, a generator
$\xi \in \ker(\theta)$ is necessariy a non-zero-divisor. Therefore,
the $\ker(\theta)$-adic topology on $A$ is complete and separated if
and only if $A$ is derived $\xi$-adically complete, and since $A$ is
derived $p$-adically complete, this, in turn, is equivalent to $A/p$
being derived $\xi$-adically complete. Now, we have
$$\begin{xy}
(-25,7)*+{ A/p }="10";
(0,7)*+{ A/(p,\xi^n) }="11";
(29,7)*+{ A/(p,\xi) }="12";
(50,7)*+{ R/p }="13";
(-25,-7)*+{ A/p }="20";
(0,-7)*+{ A/(p,\xi^{n-1}) }="21";
(29,-7)*+{ A/(p,\xi) }="22";
(50,-7)*+{ R/p }="23";
{ \ar "11";"10";};
{ \ar^-{\varphi^{-n}} "12";"11";};
{ \ar@{=} "20";"10";};
{ \ar^-{\theta} "13";"12";};
{ \ar^-{\varphi} "21";"11";};
{ \ar^-{\varphi} "22";"12";};
{ \ar^-{\varphi} "23";"13";};
{ \ar "21";"20";};
{ \ar^-{\varphi^{-(n-1)}} "22";"21";};
{ \ar^-{\theta} "23";"22";};
\end{xy}$$
with the middle and right-hand maps bijective and with the
remaining maps surjective. Taking derived limits, we obtain
$$\xymatrix{
{ A/p } \ar[r] &
{ (A/p)_{\xi}^{\wedge} } \ar[r] &
{ (A/\xi)^{\flat} } \ar[r] &
{ R^{\flat} } \cr
}$$
with the middle and right-hand maps equivalences. The composite
map takes the class of $\sum [a_i]p^i$ to $a_0$, and therefore, it,
too, is an equivalence. This proves that the left-hand map is an
equivalence, as desired. 

Let $\theta_D \colon D \to R$ be as in the statement. We wish to prove
that there is a unique ring homomorphism $f \colon A \to D$ such that
$\theta = \theta_D \circ f$. Since $A$ and $D$ are derived
$p$-complete and $L_{(A/p)/\mathbb{F}_p} \simeq 0$, this is equivalent
to showing that there is a unique ring homomorphism
$\smash{ \bar{f} \colon A/p \to D/p }$ with the property that
$\bar{\theta} = \bar{\theta}_D \circ \bar{f}$, where $\bar{\theta}
\colon A/p \to R/p$ and $\bar{\theta}_D \colon D/p \to R/p$ are
induced by $\theta \colon A \to R$ and $\theta_D \colon D \to R$,
respectively. Identifying $A/p$ with $\smash{ R^{\flat} }$, we wish to
show that there is a unique ring homomorphism
$\smash{ \bar{f} \colon R^{\flat} \to D/p }$ such that
$\smash{ a^{\#} +pR = \bar{\theta}_D(\bar{f}(a)) }$ for all
$\smash{ a \in R^{\flat} }$. Since the $\ker(\theta_D)$-adic
topology on $D$ is complete and separated, so is the
$\ker(\bar{\theta}_D)$-adic topology on $D/p$. It follows that for
$\smash{ a = (x_0,x_1, \dots ) \in R^{\flat} }$, the limit
$$\bar{f}(a) = \lim_{n \to \infty} \widetilde{x}_n^{\,p^n} \in D/p,$$ 
where we choose $\widetilde{x}_n \in D/p$ with
$\bar{\theta}_D(\widetilde{x}_n) = x_n \in R/p$, exists and is
independent of the choices made. This defines a map $\smash{ \bar{f}
  \colon R^{\flat} \to D/p }$, and the uniqueness of the limit implies
that it is a ring homomorphism. It satisfies $\bar{\theta} =
\bar{\theta}_D \circ \bar{f}$ by construction, and it is unique with
this property, since the $\ker(\bar{\theta}_D)$-adic topology on
$D/p$ is separated.
\end{proof}

We identify the diagram of $p$-adic homotopy groups
$$\begin{xy}
(0,0)*+{ \TC_*^{-}(R,\mathbb{Z}_p) }="1";
(30,0)*+{ \TP_*(R,\mathbb{Z}_p) }="2";
{ \ar@<.7ex>^-{\varphi} "2";"1";};
{ \ar@<-.7ex>_-{\operatorname{can}} "2";"1";};
\end{xy}$$
for $R$ perfectoid. By B\"{o}kstedt periodicity, the spectral sequences
$$\begin{aligned}
E_{i,j}^2 & = H^{-i}(B\mathbb{T},\THH_j(R,\mathbb{Z}_p)) \Rightarrow
\TC_{i+j}^{-}(R,\mathbb{Z}_p) \cr
E_{i,j}^2 & = \hat{H}^{-i}(B\mathbb{T},\THH_j(R,\mathbb{Z}_p))
\Rightarrow \TP_{i+j}(R,\mathbb{Z}_p) \cr
\end{aligned}$$
collapse. It follows that the respective edge homomorphisms in total
degree $0$ satisfy the hypotheses of Proposition~\ref{prop:fontaine},
and therefore, there exists a unique ring homomorphism making the
diagram
$$\xymatrix{
{ A_{\inf}(R) } \ar[r] \ar[d]^{\theta} &
{ \TC_0^{-}(R,\mathbb{Z}_p) } \ar[r]^-{\operatorname{can}}
\ar[d]^-{\operatorname{edge}} &
{ \TP_0(R,\mathbb{Z}_p) } \ar[d]^-{\operatorname{edge}} \cr
{ R } \ar@{=}[r] &
{ \THH_0(R,\mathbb{Z}_p) } \ar@{=}[r] &
{ \THH_0(R,\mathbb{Z}_p) } \cr
}$$
commute. We view $\TC_*^{-}(R,\mathbb{Z}_p)$ and
$\TP_*(R,\mathbb{Z}_p)$ as graded $A_{\inf}(R)$-algebras via the top
horizontal maps in the diagram. Bhatt--Morrow--Scholze make the
following calculation
in~\cite[Propositions~6.2 and 6.3]{bhattmorrowscholze2}.
\index{topological cyclic homology!negative}
\index{topological cyclic homology!periodic}
\index{negative topological cyclic homology!of perfectoid ring}
\index{periodic topological cyclic homology!of perfectoid ring}

\begin{theorem}[Bhatt--Morrow--Scholze]\label{thm:tpperfectoid}Let $R$
be a perfectoid ring, and let $\xi$ be a generator of the kernel of
Fontaine's map $\theta \colon A_{\inf}(R) \to R$.
\begin{enumerate}
\item[{\rm (1)}]There exists $\sigma \in \TP_2(R,\mathbb{Z}_p)$ such
  that
$$\TP_*(R,\mathbb{Z}_p) = A_{\inf}(R)[\sigma^{\pm1}].$$
\item[{\rm (2)}]There exists
$u \in \TC_2^{-}(R,\mathbb{Z}_p)$ and $v \in \TC_{-2}^{-}(R,\mathbb{Z}_p)$
such that
$$\TC_*^{-}(R,\mathbb{Z}_p) = A_{\inf}(R)[u,v]/(uv-\xi).$$
\item[{\rm (3)}]The graded ring homomorphisms
$$\begin{xy}
(0,0)*+{ \TC_*^{-}(R,\mathbb{Z}_p) }="1";
(30,0)*+{ \TP_*(R,\mathbb{Z}_p) }="2";
{ \ar@<.7ex>^-{\varphi} "2";"1";};
{ \ar@<-.7ex>_-{\operatorname{can}} "2";"1";};
\end{xy}$$
are $\varphi$-linear and $A_{\inf}(R)$-linear, respectively, and  $u$,
$v$, and $\sigma$ can be chosen in such a way that 
$\varphi(u) = \alpha \cdot \sigma$, $\varphi(v) =
\alpha^{-1}\varphi(\xi) \cdot \sigma^{-1}$,
$\operatorname{can}(u) = \xi \cdot\sigma$, and
$\operatorname{can}(v) = \sigma^{-1}$, where $\alpha \in A_{\inf}(R)$
is a unit.\footnote{\,If one is willing to replace $\xi$ by
$\varphi^{-1}(\alpha) \xi$, then the unit $\alpha$ can be
eliminated.} 
\end{enumerate}
\end{theorem}

\begin{proof}The canonical map $A_{\inf}(R) \to \TP_0(R,\mathbb{Z}_p)$
is a map of filtered rings, where the domain and target are given the
$\xi$-adic filtration and the filtration induced by the Tate spectral
sequence, respectively. Since both filtrations are complete and
separated, the map is an isomorphism if and only if the induced maps
of filtration quotients are isomorphisms. These, in turn, are
$R$-linear maps between free $R$-modules of rank $1$, and to prove
that they are isomorphisms, it suffices to consider the case
$R = \mathbb{F}_p$.

The canonical map
$\operatorname{can} \colon \TC_*^{-}(R,\mathbb{Z}_p)
\to \TP_*(R,\mathbb{Z}_p)$ induces the map of spectral sequences that, on
$E^2$-terms, is given by the localization map
$$\xymatrix{
{ R[t,x] } \ar[r] &
{ R[t^{\pm1},x], } \cr
}$$
where $t \in E_{-2,0}^2$ and $x \in E_{0,2}^2$ are any $R$-module
generators. It follows that $\TP_*(R,\mathbb{Z}_p)$ is 2-periodic and
concentrated in even degrees, so~(1) holds for any
$\sigma \in \TP_2(R,\mathbb{Z}_p)$ that is an $A_{\inf}(R)$-module
generator, or equivalently, is represented in the spectral sequence by
an $R$-module generator of $E_{2,0}^2$. We now fix a choice of
$\sigma \in \TP_2(R,\mathbb{Z}_p)$, and let $t \in E_{-2,0}^2$ and
$x \in E_{0,2}^2$ be the unique elements that 
represent $\sigma^{-1} \in \TP_{-2}(R,\mathbb{Z}_p)$ and
$\xi\sigma \in \TP_2(R,\mathbb{Z}_p)$, respectively. The latter
classes are the images by the canonical map of unique classes $v \in
\TC_{-2}^{-}(R,\mathbb{Z}_p)$ and $u \in \TC_2^{-}(R,\mathbb{Z}_p)$,
and $uv = \xi$. This proves~(2) and the part of~(3) that concerns the
canonical map.

It remains to identify $\varphi \colon \TC_*^{-}(R,\mathbb{Z}_p) \to
\TP_*(R,\mathbb{Z}_p)$. In degree zero, we have fixed identifications
of domain and target with $A_{\inf}(R)$, and we first prove that, with
respect to these identifications, the map in question is given by the
Frobenius $\varphi \colon A_{\inf}(R) \to A_{\inf}(R)$. To this end,
we consider the diagram
$$\begin{xy}
(0,7)*+{ \THH(R)^{h\mathbb{T}} }="11";
(37,7)*+{ (\THH(R)^{tC_p})^{h\mathbb{T}} }="12";
(73,7)*+{ (R^{tC_p})^{h\mathbb{T}} \phantom{,()}}="13";
(0,-7)*+{ \THH(R) }="21";
(37,-7)*+{ \THH(R)^{tC_p} }="22";
(73,-7)*+{ R^{tC_p}, \phantom{()} }="23";
{ \ar^-{\varphi^{h\mathbb{T}}} "12";"11";};
{ \ar "13";"12";};
{ \ar "21";"11";};
{ \ar "22";"12";};
{ \ar@<-1ex> "23";"13";};
{ \ar^-{\varphi} "22";"21";};
{ \ar "23";"22";};
\end{xy}$$
where, on the right, we view $R$ as an $\mathbb{E}_{\infty}$-ring with
trivial $\mathbb{T}$-action, and where the right-hand horizontal maps both
are induced by the unique extension $\THH(R) \to R$ of the identity
map of $R$ to a map of $\mathbb{E}_{\infty}$-rings with
$\mathbb{T}$-action. Applying $\pi_0(-,\mathbb{Z}_p)$, we obtain the
diagram of rings
$$\begin{xy}
(0,7)*+{ A_{\inf}(R) }="11";
(28,7)*+{ A_{\inf}(R) }="12";
(55,7)*+{ R\phantom{,} }="13";
(0,-7)*+{ R }="21";
(28,-7)*+{ R/p }="22";
(55,-7)*+{ R/p, }="23";
{ \ar^-{\pi_0(\varphi^{h\mathbb{T}})} "12";"11";};
{ \ar^-{\theta} "13";"12";};
{ \ar^-{\theta} "21";"11";};
{ \ar "22";"12";};
{ \ar^-{\pr} "23";"13";};
{ \ar^-{\pi_0(\varphi)} "22";"21";};
{ \ar@{=} "23";"22";};
\end{xy}$$
where we identify the upper right-hand horizontal map by applying
naturality of the edge homomorphism of the Tate spectral sequence 
to $\THH(R) \to R$. Now, it follows from the proof of
Proposition~\ref{prop:fontaine} that the map
$\pi_0(\varphi^{h\mathbb{T}})$ is uniquely determined by the map
$\pi_0(\varphi)$. Moreover, the latter map is identified
in~\cite[Corollary~IV.2.4]{nikolausscholze} to be the map 
that takes $x$ to the class of $x^p$. We conclude that
$\pi_0(\varphi^{h\mathbb{T}})$ is equal to the Frobenius
$\varphi \colon A_{\inf}(R) \to A_{\inf}(R)$, since the latter makes
the left-hand square commute. Finally, since
$$\xymatrix{
{ \TC_2^{-}(R,\mathbb{Z}_p) } \ar[r]^-{\varphi} &
{ \TP_2(R,\mathbb{Z}_p) } \cr
}$$
is an isomorphism, we have $\varphi(u) = \alpha \cdot \sigma$ with
$\alpha \in A_{\inf}(R)$ a unit, and the relation $uv = \xi$ implies
that $\varphi(v) = \alpha^{-1}\varphi(\xi) \cdot \sigma^{-1}$.
\end{proof}

\index{perfectoid ring|)}

\subsection{Bott periodicity}\label{section:bottperiodicity}
\index{Bott periodicity|(}

We fix a field $C$ that contains $\mathbb{Q}_p$ and that is both
algebraically closed and complete with respect to a non-archimedean
absolute value that extends the $p$-adic absolute value on
$\mathbb{Q}_p$. The valuation ring\index{$\mathcal{O}_C$}
$$\mathcal{O}_C = \{ x \in C \mid \lvert x \rvert
\leq 1\} \subset C$$
is a perfectoid ring, and we proceed to evaluate its topological
cyclic homology. We explain that this calculation, which uses
B\"{o}kstedt periodicity, gives a purely $p$-adic proof of Bott
periodicity.

The calculation uses a particular choice of generator $\xi$ of the
kernel of the map
$\theta \colon A_{\inf}(\mathcal{O}_C) \to \mathcal{O}_C$, which we
define first. We fix a generator
$$\epsilon \in \mathbb{Z}_p(1) = T_p(C^*) =
\Hom(\mathbb{Q}_p/\mathbb{Z}_p,C^*)$$
of the $p$-primary Tate module of $C^*$. It determines and is
determined by the sequence $(1,\zeta_p,\zeta_{p^2}, \dots)$ of
compatible primitive $p$-power roots of unity in $C$, where
$\zeta_{p^n} = \epsilon(\nicefrac{1}{p^n} + \mathbb{Z}_p) \in C$. By
abuse of notation, we also write
$$\epsilon = (1,\zeta_p,\zeta_{p^2}, \dots ) \in \mathcal{O}_{C^{\flat}}.$$
We now define the elements $\mu,\xi \in A_{\inf}(\mathcal{O}_C)$
by\index{$\mu$}
$$\begin{aligned}
\mu & = [\epsilon] - 1 \cr
\xi & = \mu/\varphi^{-1}(\mu)
= ([\epsilon]-1)/([\epsilon^{1/p}]-1). \cr
\end{aligned}$$
The element $\xi$ lies in the kernel of $\theta \colon
A_{\inf}(\mathcal{O}_C) \to \mathcal{O}_C$, since
$$\textstyle{ \theta(\xi) = \theta(\sum_{0 \leq k < p}
  [\epsilon^{1/p}]^k) = \sum_{0 \leq k < p} \zeta_p^k = 0, }$$
and it is a generator. More generally, the element
$$\xi_r = \mu/\varphi^{-r}(\mu) =
([\epsilon]-1)/([\epsilon^{1/p^r}]-1)$$
generates the kernel of $\theta_r \colon A_{\inf}(\mathcal{O}_C) \to
W_r(\mathcal{O}_C)$, and the element $\mu$ generates the
kernel of the induced map\footnote{\,Its cokernel
$R^1\lim_r\ker(\theta_r)$ is a huge $A_{\inf}(\mathcal{O}_C)$-module
that is almost zero.}
$$\begin{xy}
(0,0)*+{ A_{\inf}(\mathcal{O}_C) }="1";
(40,0)*+{ \lim_{r,R} W_r(\mathcal{O}_C) = W(\mathcal{O}_C). }="2";
{ \ar^-{(\theta_r)} "2";"1";};
\end{xy}$$

\begin{theorem}\label{thm:tcoc}
\index{topological cyclic homology!of $\mathcal{O}_C$}
\index{$\mathcal{O}_C$!topological cyclic homology of}
Let $C$ be a field extension of
$\mathbb{Q}_p$ that is both algebraically closed and complete with
respect to a non-archimedean absolute value that extends the
$p$-adic absolute value on $\mathbb{Q}_p$, and let
$\mathcal{O}_C \subset C$ be the valuation ring. The canonical map of
graded rings
$$\xymatrix{
{ \operatorname{Sym}_{\hskip.5pt\mathbb{Z}_p}(\TC_2(\mathcal{O}_C,\mathbb{Z}_p)) }
\ar[r] &
{ \TC_*(\mathcal{O}_C,\mathbb{Z}_p) } \cr
}$$
is an isomorphism, and $\TC_2(\mathcal{O}_C,\mathbb{Z}_p)$ is a free
$\mathbb{Z}_p$-module of rank $1$. Moreover, the map
$\TC_{2m}(\mathcal{O}_C,\mathbb{Z}_p) \to
\TC_{2m}^{-}(\mathcal{O}_C,\mathbb{Z}_p)$ takes a
$\mathbb{Z}_p$-module generator of the domain to $\varphi^{-1}(\mu)^m$
times an $A_{\inf}(\mathcal{O}_C)$-module generator of the target.
\end{theorem}

\begin{proof}We let $\epsilon \in T_p(C^*)$ and
$\xi ,\mu \in A_{\inf}(\mathcal{O}_C)$ be as above. According to
Theorem~\ref{thm:tpperfectoid}, the even (resp.~odd) $p$-adic homology
groups of $\TC(\mathcal{O}_C)$ are given by the kernel
(resp.~cokernel) of the $\mathbb{Z}_p$-linear map
$$\begin{xy}
(0,0)*+{ A_{\inf}(\mathcal{O}_C)[u,v]/(uv-\xi) }="1";
(44,0)*+{ A_{\inf}(\mathcal{O}_C)[\sigma^{\pm1}] }="2";
{ \ar^-{\varphi-\operatorname{can}} "2";"1";};
\end{xy}$$
given by
$$\begin{aligned}
(\varphi - \operatorname{can})(a \cdot u^m) 
{} & = (\alpha^m\varphi(a) - \xi^m a) \cdot \sigma^m \cr
(\varphi - \operatorname{can})(a \cdot v^m)
{} & = (\alpha^{-m}\varphi(\xi)^m\varphi(a) - a) \cdot
\sigma^{-m}, \cr
\end{aligned}$$
where $m \geq 0$ is an integer, and where
$\alpha \in A_{\inf}(\mathcal{O}_C)$ is a fixed unit. 
We need only consider the top formula, since we know, for general
reasons, that the $p$-adic homotopy groups of $\TC(\mathcal{O}_C)$ are
concentrated in degrees $\geq -1$. 

We first prove that an element $a \cdot u^m$ in the image of the map
$$\xymatrix{
{ \TC_{2m}(\mathcal{O}_C,\mathbb{Z}_p) } \ar[r] &
{ \TC_{2m}^{-}(\mathcal{O}_C,\mathbb{Z}_p) } \cr
}$$
is of the form $a = \varphi^{-1}(\mu)^mb$, for some $b \in
A_{\inf}(\mathcal{O}_C)$. The element $b$ is uniquely determined by
$a$, since $\mathcal{O}_C$, and hence, $A_{\inf}(\mathcal{O}_C)$ is an
integral domain. This image consists of the elements $a \cdot u^m$,
where $a \in A_{\inf}(\mathcal{O}_C)$ satisfies
$$\alpha^m\varphi(a) = \xi^ma.$$
Rewriting this equation in the form
$$a = \varphi^{-1}(\alpha)^{-m}\varphi^{-1}(\xi)^m\varphi^{-1}(a),$$
we find inductively that for all $r \geq 1$,
$$a = \varphi^{-1}(\alpha_r)^{-m}\varphi^{-1}(\xi_r)^m\varphi^{-r}(a),$$
where $\alpha_r = \prod_{0 \leq i < r}\varphi^{-i}(\alpha)$ and
$\xi_r = \prod_{0 \leq i < r}\varphi^{-i}(\xi)$. Therefore, we have
$$\textstyle{
\varphi(a) \in \bigcap_{r \geq 1} \xi_r^mA_{\inf}(\mathcal{O}_C) =
\mu^mA_{\inf}(\mathcal{O}_C), }$$
as desired. Moreover, the element
$a \cdot u^m = \varphi^{-1}(\mu)^mb \cdot u^m$ belongs to the image of
the map above if and only if $b \in A_{\inf}(\mathcal{O}_C)$ solves
the equation
$$\alpha^m \varphi(b) = b.$$
Indeed, the elements $\mu,\xi \in A_{\inf}(\mathcal{O}_C)$ satisfy
$\mu = \xi\varphi^{-1}(\mu)$.

To complete the proof, it suffices to show that the canonical map
$$\xymatrix{
{ \operatorname{Sym}_{\hskip.5pt\mathbb{Z}/p}(\TC_2(\mathcal{O}_C,\mathbb{Z}/p)) }
\ar[r] &
{ \TC_*(\mathcal{O}_C,\mathbb{Z}/p) } \cr
}$$
is an isomorphism, and that the $\mathbb{Z}/p$-vector space
$\TC_2(\mathcal{O}_C,\mathbb{Z}/p)$ has dimension $1$. First, to show
that $\TC_{2m-1}(\mathcal{O}_C,\mathbb{Z}/p)$ is zero, we must show
that for every $c \in \mathcal{O}_{C^{\flat}}$, there exists $a \in
\mathcal{O}_{C^{\flat}}$ such that
$$\alpha^ma^p - \xi^ma = c.$$
The ring $\mathcal{O}_{C^{\flat}}$ is complete with respect to the
non-archimedean absolute value given by $\lvert x \rvert_{C^{\flat}} =
\lvert x^{\#}\rvert_C$, and its quotient field $C^{\flat}$ is
algebraically closed. Hence, there exists $a \in C^{\flat}$ that
solves the equation in question, and we must show that
$\lvert a \rvert_{C^{\flat}} \leq 1$. If
$\lvert a \rvert = \lvert a \rvert_{C^{\flat}} > 1$, then
$$\lvert \alpha^ma^p \rvert = \lvert a \rvert^p > \lvert a \rvert \geq
\lvert \xi \rvert^m \lvert a \rvert = \lvert \xi^ma \rvert,$$
and since $\lvert \phantom{x} \rvert$ is non-archimedean, we conclude
that
$$\lvert \alpha^ma^p - \xi^m a\rvert = \lvert \alpha^ma^p \rvert > 1
\geq \lvert c \rvert.$$
Hence, every solution $a \in C^{\flat}$ to the equation in
question is in $\mathcal{O}_{C^{\flat}}$. This shows that the group
$\TC_{2m-1}(\mathcal{O}_C,\mathbb{Z}/p)$ is zero.

Finally, we determine the $\mathbb{Z}/p$-vector space 
$\TC_{2m}(\mathcal{O}_C,\mathbb{Z}/p)$, which we have identified with
the subspace of $\TC_{2m}^{-}(\mathcal{O}_C,\mathbb{Z}/p)$ that
consists of the elements of the form
$b \varphi^{-1}(\mu)^m \cdot u^m$, where
$b \in \mathcal{O}_{C^{\flat}}$ satisfies the equation 
$$\alpha^mb^{\hskip.8pt p}  = b.$$
Since $C^{\flat}$ is algebraically closed, there are $p$ solutions to
this equations, namely, $0$ and the $(p-1)$th roots of $\alpha^{-m}$,
all of which are units in $\mathcal{O}_{C^{\flat}}$. This shows that
$\TC_{2m}(\mathcal{O}_C,\mathbb{Z}/p)$ is a $\mathbb{Z}/p$-vector
space of dimension $1$, for all $m \geq 0$. It remains only to notice
that if $b_1$ and $b_2$ satisfy $\alpha^{m_1}b_1^{\hskip.8pt p} = b_1$ and
$\alpha^{m_2}b_2^{\hskip.8pt p} = b_2$, respectively, then $b = b_1b_2$ satisfies
$\alpha^{m_1+m_2}b^{\hskip.8pt p} = b$, which shows that the map in
the statement is an isomorphism.
\end{proof}

We thank Antieau--Mathew--Morrow for sharing the elegant proof of
the following statement with us.

\begin{lemma}\label{lem:bottperiodicity}With notation as in
Theorem~\ref{thm:tcoc}, the
map\index{$\mathcal{O}_C$!$K$-theory of}
$$\xymatrix{
{ K(\mathcal{O}_C,\mathbb{Z}_p) } \ar[r]^-{j^*} &
{ K(C,\mathbb{Z}_p) } \cr
}$$
is an equivalence.
\end{lemma}

\begin{proof}For every ring $R$, the category of coherent $R$-modules
is abelian, and we define $K'(R)$ to be the algebraic $K$-theory of this
abelian category. If $R$ is coherent as an $R$-module, then the
category of coherent $R$-modules contains the category of finite
projective $R$-modules as a full exact subcategory. So in this situation
the canonical inclusion induces a map of $K$-theory spectra
$$\xymatrix{
{ K(R) } \ar[r] &
{ K'(R). } \cr
}$$
If, in addition, the ring $R$ is of finite global dimension, then the
resolution theorem~\cite[Theorem~3]{quillen} shows that this map is an
equivalence. In particular, this is so, if $R$ is a valuation
ring. For every finitely generated ideal in $R$ is principal, so $R$
is coherent, and it follows form~\cite{berstein} that $R$ is of
finite global dimension.

We now choose any pseudouniformizer $\pi \in \mathcal{O}_C$ and
apply the localization theorem~\cite[Theorem~5]{quillen} to the
abelian category of coherent $\mathcal{O}_C$-modules and the full
abelian subcategory of coherent $\mathcal{O}_C/\pi$-modules. The
localization sequence then takes the form
$$\xymatrix{
{ K'(\mathcal{O}_C/\pi) } \ar[r]^-{i_*} &
{ K(\mathcal{O}_C) } \ar[r]^-{j^*} &
{ K(C), }
}$$
since $\mathcal{O}_C$ and $C$ both are valuation rings. In a similar
manner, we obtain, for any pseudouniformizer $\pi^{\flat} \in
\mathcal{O}_{C^{\flat}}$, the localization sequence
$$\xymatrix{
{ K'(\mathcal{O}_{C^{\flat}}/\pi^{\flat}) } \ar[r]^-{i_*} &
{ K(\mathcal{O}_{C^{\flat}}) } \ar[r]^-{j^*} &
{ K(C^{\flat}). } \cr
}$$
We may choose $\pi$ and $\pi^{\flat}$ such that 
$\mathcal{O}_C/\pi$ and $\mathcal{O}_{C^{\flat}}/\pi^{\flat}$
are isomorphic rings, so we conclude that the lemma is equivalent to
the statement that
$$\xymatrix{
{ K(\mathcal{O}_{C^{\flat}},\mathbb{Z}_p) } \ar[r]^-{j^*} &
{ K(C^{\flat},\mathbb{Z}_p) } \cr
}$$
is an equivalence. But $\mathcal{O}_{C^{\flat}}$ and $C^{\flat}$ are both
perfect local $\mathbb{F}_p$-algebras, so the domain and target are
both equivalent to $\mathbb{Z}_p$. Finally, the map in question is
a map of $\mathbb{E}_{\infty}$-algebras in spectra, so it is
necessarily an equivalence.
\end{proof}

\begin{corollary}[Bott periodicity]\label{cor:bottperiodicity}The
canonical map of graded rings
$$\xymatrix{
{ \operatorname{Sym}_{\hskip.5pt\mathbb{Z}}(K_2^{\operatorname{top}}(\mathbb{C})) } \ar[r] &
{ K_*^{\operatorname{top}}(\mathbb{C}) } \cr
}$$
is an isomorphism, and $K_2^{\operatorname{top}}(\mathbb{C})$ is a
free abelian group of rank $1$.
\end{corollary}

\begin{proof}The homotopy groups of
$K^{\operatorname{top}}(\mathbb{C})$ are finitely
generated\footnote{\,This follows by a Serre class argument from the
  fact that the homology groups of the underlying space are 
  finitely generated.}, so it 
suffices to prove the analogous statements for the $p$-adic 
homotopy groups, for all prime numbers $p$. We fix $p$, let
$C$ be as in Theorem~\ref{thm:tcoc}, and choose a ring homomorphism
$f \colon \mathbb{C} \to C$. By Suslin~\cite{suslin,suslin1}, the
canonical maps
$$\xymatrix{
{ K^{\operatorname{top}}(\mathbb{C}) } &
{ K(\mathbb{C}) } \ar[l] \ar[r]^-{f^*} &
{ K(C) } \cr
}$$
become weak equivalences upon $p$-completion. Moreover, by
Lemma~\ref{lem:bottperiodicity}, the map
$j^* \colon K(\mathcal{O}_C) \to K(C)$ becomes a weak equivalence
after $p$-completion. The ring $\mathcal{O}_C$ is a henselian local
ring with algebraically closed residue field $k$ of characteristic
$p$. Therefore, by Clausen--Mathew--Morrow~\cite{clausenmathewmorrow},
the cyclotomic trace map induces an
equivalence\index{Clausen--Mathew--Morrow theorem}
$$\xymatrix{
{ K(\mathcal{O}_C,\mathbb{Z}_p) } \ar[r]^-{\tr} &
{ \TC(\mathcal{O}_C,\mathbb{Z}_p), } \cr
}$$
so the statement follows from Theorem~\ref{thm:tcoc}.
\end{proof}
\index{Bott periodicity|)}

\section{Group rings}\label{sec:grouprings}
\index{group ring!topological cyclic homology of|(}
\index{topological cyclic homology!of group ring|(}

Let $G$ be a discrete group. We would like to understand the
topological cyclic homology of the group ring $R[G]$, where $R$ is a
ring or, more generally, a connective $\mathbb{E}_1$-algebra in
spectra. Since the assignment $G \mapsto \TC(R[G])$ is functorial in
the 2-category of groups,\footnote{By the latter we mean the full
  subcategory of the $\infty$-category of spaces consisting of spaces
  of the form $BG$. Concretely objects are groups, morphisms are group
  homomorphisms and 2-morphisms are conjugations.} we get an
``assembly'' map\index{assembly map}
$$\xymatrix{
{ \TC(R) \otimes BG_+ } \ar[r] &
{ \TC(R[G]), } \cr
}$$
and what we will actually do here is to consider the cofiber of this
map. By~\cite[Theorem~1.2]{luckreichrognesvarisco}, topological cyclic homology for a
given group $G$ can be assembled from the 
cyclic subgroups of $G$. We will focus on the case $G = C_p$ of a
cyclic group of prime order $p$, but the methods that we present here
can be generalized to the case  of cyclic $p$-groups. We
will be interested in the $p$-adic homotopy type of the cofiber. To
this end, we remark that the $p$-completion of $\THH(R)$, which we
denote by $\THH(R,\mathbb{Z}_p)$ as before, inherits a cyclotomic structure. For $R$ connective, we have 
$$\TC(R, \mathbb{Z}_p) \simeq \TC(\THH(R, \mathbb{Z}_p)).$$
The formula we give involves the non-trivial extension of groups
$$\xymatrix{
{ C_p } \ar[r] &
{ \mathbb{T}_p } \ar[r] &
{ \mathbb{T}. } \cr
}$$
The middle group $\mathbb{T}_p$ is a circle, but the right-hand
map is a $p$-fold cover, and by restriction along this map, a spectrum
with $\mathbb{T}$-action $X$ gives rise to a spectrum with
$\mathbb{T}_p$-action, which we also write $X$.

\begin{theorem}\label{thm:groupring}For a connective
$\mathbb{E}_1$-algebra $R$ in spectra, there is a natural cofiber
sequence of spectra
$$\xymatrix{
{ \TC(R,\mathbb{Z}_p) \otimes BC_{p+} } \ar[r] &
{ \TC(R[C_p],\mathbb{Z}_p) } \ar[r] &
{ \THH(R,\mathbb{Z}_p)_{h\mathbb{T}_p}[1] \otimes C_p, } \cr
}$$
where $C_p$ is considered as a pointed set with basepoint $1$. 
\end{theorem}

We note that the right-hand term in the sequence above is
non-canonically equivalent to a $(p-1)$-fold sum of copies of
$\THH(R,\mathbb{Z}_p)_{h\mathbb{T}_p}[1]$. We do not determine the
boundary map in the sequence. The proof of Theorem \ref{thm:groupring}
requires some preparation and preliminary results. First, we recall
that for a spherical group ring 
$\mathbb{S}[G]$, there is a natural equivalence
$$\THH(\mathbb{S}[G]) \simeq \mathbb{S} \otimes LBG_+,$$
where $LBG = \mathrm{Map}(\mathbb{T}, BG)$ is the free loop
space. Moreover, the
equivalence is $\mathbb{T}$-equivariant for the $\mathbb{T}$-action on
$LBG$ induced from the action of $\mathbb{T}$ on itself by
multiplication. Hence, for general $R$, we have
$$\THH(R[G]) \simeq \THH(R) \otimes LBG_+,$$
where $\mathbb{T}$ acts diagonally on the right-hand side. Since $G$
is discrete, we further have a $\mathbb{T}$-equivariant decomposition
of spaces
$$\textstyle{ LBG \simeq \coprod BC_G(x), }$$
where $x$ ranges over a set of representatives of the conjugacy
classes of elements of $G$, and where $C_G(x) \subset G$ is the
centralizer of $x \in G$. The action by
$\mathbb{T} \simeq B\mathbb{Z}$ on $BC_G(x)$ is given by the map
$B\mathbb{Z} \times BC_G(x) \to BC_G(x)$ induced by the group
homomorphism $\mathbb{Z} \times C_G(x) \to C_G(x)$ that to $(n,g)$
assigns $x^ng$. Specializing to the case $G = C_p$, we get the
following description.

\begin{lemma}\label{lem:freeloopspace}There is a natural
$\mathbb{T}$-equivariant cofiber sequence of spaces
\[
BC_p^{\,\operatorname{triv}}  \to LBC_p \to 
(BC_p^{\,\operatorname{res}}) _+\otimes C_p,
\]
where $BC_p^{\,\operatorname{triv}}$ has the trivial $\mathbb{T}$-action and 
$BC_p^{\,\operatorname{res}} \simeq (\mathrm{pt})_{hC_p}$ has the
residual action by $\mathbb{T} = \mathbb{T}_p/C_p$. \footnote{This
comes from the fact that $\operatorname{pt}$ carries a (necessarily
trivial) $\mathbb{T}_p$-action. In a point set model, it can be
described as $E\hskip.7pt\mathbb{T}_p / C_p$.} \end{lemma}

As a consequence, we obtain for every $\mathbb{E}_1$-algebra in spectra
$R$, a cofiber sequence of spectra with $\mathbb{T}$-action 
$$\xymatrix{
{ \THH(R) \otimes BC_{p+}^{\,\operatorname{triv}} } \ar[r] &
{ \THH(R[G]) } \ar[r] &
{ (\THH(R) \otimes BC_{p+}^{\,\operatorname{res}}) \otimes C_p, } \cr
}$$
where the left-hand map is the assembly map. To determine the
cyclotomic structure on the terms of this sequence, we prove the
following result.

\begin{lemma}\label{lem:tateorbitadvanced}Let $X$ be a spectrum with
$\mathbb{T}$-action that is bounded below, and let $\mathbb{T}$ act
diagonally on $X \otimes BC_{p+}^{\,\operatorname{res}}$. Then $(X
\otimes  BC_{p+}^{\,\operatorname{res}})^{tC_p} \simeq 0$.
\end{lemma}

\begin{proof}We write $X \simeq \lim_n \tau_{\leq n}X$ as the limit of
its Postnikov tower. The spectra $\tau_{\leq  n} X$ inherit a
$\mathbb{T}$-action, and the equivalence is
$\mathbb{T}$-equivariant. The map induced by the canonical projections,
$$\xymatrix{
{ X \otimes BC_{p+}^{\,\operatorname{res}} } \ar[r] &
{ \lim_n ( \tau_{\leq n} X \otimes BC_{p+} ), } \cr
}$$
is an equivalence, since the connectivity of the fibers tend to
infinity with $n$, and therefore, also the map 
$$\xymatrix{
{ (X \otimes BC_{p+}^{\,\operatorname{res}})^{tC_p} } \ar[r] &
{ \lim_n ((\tau_{\leq n} X \otimes BC_{p+})^{tC_p} ) } \cr
}$$
is an equivalence. Indeed, the analogous statements for homotopy
fixed points and homotopy orbits is respectively clear and a
consequence of the fact that the connectivity of the fibers tend to
infinity with $n$.

Since $X$ is bounded below,  we may assume that $X$ is
concentrated in a single degree with necessarily trivial
$\mathbb{T}$-action. As spectra with $\mathbb{T}$-action,
$$X \otimes BC_{p+}^{\,\operatorname{res}} \simeq X_{hC_p},$$
where the right-hand side has the residual $\mathbb{T}$-action. But
$(X_{hC_p})^{tC_p} \simeq 0$ by the Tate orbit lemma~\cite[Lemma
I.2.1]{nikolausscholze}, so the lemma follows.\index{Tate orbit lemma}
\end{proof}

\begin{corollary}\label{cor:tateorbitadvanced}Let $X$ be a spectrum
with $\mathbb{T}$-action that is bounded below and $p$-complete. The
spectrum $X \otimes BC_{p+}^{\,\operatorname{res}}$ with the diagonal
$\mathbb{T}$-action admits a unique cyclotomic structure, and,
with respect to this structure,
$$\TC(X \otimes BC_{p+}^{\,\operatorname{res}}) \simeq  X_{h\mathbb{T}_p}[1].$$
\end{corollary}

\begin{proof}A cyclotomic structure on a spectrum with
$\mathbb{T}$-action $Y$ consists of a family of
$\mathbb{T}$-equivariant maps $\varphi_{\ell} \colon Y \to
Y^{tC_{\ell}}$, one for every prime number $\ell$, including $p$.
In the case of $Y = X \otimes BC_p^{\,\operatorname{res}}$, the target of
this map is contractible for all $\ell$. Indeed, for $\ell \neq p$,
this follows from $Y$ being $p$-complete, and for $\ell = p$, it
follows from Lemma~\ref{lem:tateorbitadvanced}. Hence, there is a
unique such family of maps. In order to evaluate $\TC$, we first
note that Lemma~\ref{lem:tateorbitadvanced} also implies that
$$\TP(X \otimes BC_{p+}^{\,\operatorname{res}}) 
\simeq \TP(X \otimes BC_{p+}^{\,\operatorname{res}})^{\wedge}
\simeq 0.$$
Accordingly,
$$\TC(X \otimes BC_{p+}^{\,\operatorname{res}}) 
\simeq \TC^{-}(X \otimes BC_{p+}^{\,\operatorname{res}})
= ( X \otimes BC_{p+}^{\,\operatorname{res}})^{h\mathbb{T}},$$
and by the vanishing of $(-)^{t\hskip.5pt\mathbb{T}}$, this is further equivalent to
$$( X \otimes BC_{p+}^{\,\operatorname{res}})_{h\mathbb{T}}[1]
\simeq (X_{hC_p})_{h\mathbb{T}}[1] \simeq X_{h\mathbb{T}_p}[1],$$
where, in the middle term, $C_p \subset \mathbb{T}_p$ acts trivially
on $X$.
\end{proof}

\begin{proof}[Proof of Theorem~\ref{thm:groupring}]We have proved that
  there is a cofiber sequence
$$\xymatrix{
{ \THH(R) \otimes BC_{p+}^{\,\operatorname{triv}} } \ar[r] &
{ \THH(R[C_p]) } \ar[r] &
{ \THH(R) \otimes BC_{p+}^{\,\operatorname{res}} \otimes C_p } \cr
}$$
of spectra with $\mathbb{T}$-action in which the left-hand map is the
assembly map. We get an induced fiber sequence with $p$-adic
coefficients. Since the forgetful functor from cyclotomic spectra to spectra 
with $\mathbb{T}$-action creates colimits, this map is also the
assembly map in the $\infty$-category of cyclotomic spectra,
and by Corollary~\ref{cor:tateorbitadvanced}, the induced cyclotomic
structure on cofiber necessarily is the unique cyclotomic structure,
for which
$$\TC( \THH(R, \mathbb{Z}_p) \otimes BC_{p+}^{\,\operatorname{res}}
\otimes C_p )
\simeq \THH(R, \mathbb{Z}_p)_{h\mathbb{T}_p}[1] \otimes C_p.$$
Finally, by the universal property of the colimit, there is a canonical map
$$\xymatrix{
{ \TC(R,\mathbb{Z}_p) \otimes BC_p } \ar[r] &
{ \TC( \THH(R,\mathbb{Z}_p) \otimes BC_p^{\,\operatorname{triv}}), } \cr
}$$
and by~\cite[Theorem 2.7]{clausenmathewmorrow}, this map is an equivalence. 
\end{proof}

\begin{example}For $R = \mathbb{S}$, the sequence in
Theorem~\ref{thm:groupring} becomes
$$\xymatrix{
{ \TC(\mathbb{S}, \mathbb{Z}_p) \otimes BC_{p+} } \ar[r] &
{ \TC(\mathbb{S}[C_p], \mathbb{Z}_p) } \ar[r] &
{ (\mathbb{S}_p \otimes B\mathbb{T}_{p+})[1] \otimes C_p, } \cr
}$$
where $B\mathbb{T}_p \simeq \mathbb{P}^{\infty}(\mathbb{C})$. One can
also give a formula for $\TC(\mathbb{S}[C_p], \mathbb{Z}_p) $ in this
case, but this is more complicated than the formula for the cofiber of the
assembly map.  
\end{example}

Finally, we evaluate the homotopy groups of
$\THH(R,\mathbb{Z}_p)_{h\mathbb{T}_p}$ in the case, where $R$ is a
$p$-torsion free perfectoid ring. We consider the diagram
$$\xymatrix{
{ \pi_0(\THH(R,\mathbb{Z}_p)^{h\mathbb{T}}) } \ar[r] \ar[d] &
{ \pi_0(\THH(R,\mathbb{Z}_p)^{h\mathbb{T}_p})\phantom{,} } \ar[d] \cr
{ \pi_0(\THH(R,\mathbb{Q}_p)^{h\mathbb{T}}) } \ar[r] &
{ \pi_0(\THH(R,\mathbb{Q}_p)^{h\mathbb{T}_p}), } \cr
}$$
where the horizontal maps are given by restriction along
$\mathbb{T}_p \to \mathbb{T}$, and where the vertical maps are given
by change-of-coefficients. The respective homotopy fixed point
spectral sequences endow each of the four rings with a descending
filtration, which we refer to as the Nygaard filtration, and they are
all complete and separated in the topology.\index{Nygaard filtration}
We have identified the top
left-hand ring with Fontaine's ring $A = A_{\inf}(R)$. The lower
horizontal map is an isomorphism, and the common ring is 
identified with $A[\nicefrac{1}{p}]^{\wedge}$, where
``$(-)^{\wedge}$'' indicates Nygaard completion. We further have
compatible edge homomorphisms
$$\xymatrix{
{ \pi_0(\THH(R,\mathbb{Z}_p)^{h\mathbb{T}}) } \ar[r]
\ar[d]^-{\theta_{\phantom{p}}} &
{ \pi_0(\THH(R,\mathbb{Z}_p)^{h\mathbb{T}_p})\phantom{,} }
\ar[d]^-{\theta_p} \cr
{ \THH_0(R,\mathbb{Z}_p) } \ar@{=}[r] &
{ \THH_0(R,\mathbb{Z}_p), } \cr
}$$
whose kernels $I \subset A$ and $I_p \subset A_p$ are principal
ideals, and we can choose generators $\xi$ and $\xi_p$ such that the
top horizontal map takes $\xi$ to $p\hskip.5pt\xi_p$. This identifies
the top right-hand ring $A_p$ with the subring
$$\textstyle{ A_p = (\sum_{n \geq 0}p^{-n}I^n)^{\wedge} \subset
A[\nicefrac{1}{p}]^{\wedge} }$$
given by the Nygaard completion of the Rees
construction.\index{Rees construction!$p$-adic} 

\begin{proposition}\label{lem:pfoldcoverperfectoid}Let $R$ be a $p$-torsion
free perfectoid ring. The map
$$\xymatrix{
{ \pi_*(\THH(R,\mathbb{Z}_p)^{h\mathbb{T}_p}) } \ar[r]^-{\operatorname{can}} &
{ \pi_*(\THH(R,\mathbb{Z}_p)^{t\mathbb{T}_p}) } \cr
}$$
is given by the localization of graded $A_p$-algebras
$$\xymatrix{
{ A_p[u,v_p]/(uv_p - \xi_p) } \ar[r] &
{ A_p[u,v_p^{\pm1}]/(uv_p - \xi_p) = A_p[v_p^{\pm1}], } \cr
}$$
where $u$ and $v_p$ are homogeneous elements of degree $2$ and $-2$,
and where $\xi_p$ is a generator of the kernel $I_p$ of the edge
homomorphism $\theta_p \colon A_p \to R$.
\end{proposition}

\begin{proof}The proof is analogous to the proof of Theorem~\ref{thm:tpperfectoid}.
\end{proof}

\begin{corollary}\label{lem:pfoldcoverperfectoid}Let $R$ be a $p$-torsion
free perfectoid ring. For $m \geq 0$,
$$\pi_{2m}(\THH(R,\mathbb{Z}_p)_{h\mathbb{T}_p}) = A_p/I_p^{m+1} \cdot v_p^{-(m+1)},$$
and the remaining homotopy groups are zero.
\end{corollary}

\begin{remark}\label{rem:pfoldcoverperfect}It is interesting to
compare the calculation above to the case, where $R$ is a perfect
$\mathbb{F}_p$-algebra. In this case, the map
$$\xymatrix{
{ \pi_0(\THH(R,\mathbb{Z}_p)^{h\mathbb{T}}) } \ar[r] &
{ \pi_0(\THH(R,\mathbb{Z}_p)^{h\mathbb{T}_p}) } \cr
}$$
takes $uv - p = 0$ to $(uv_p-1)p = 0$, and since $1-uv_p$ is a unit,
we conclude that, in the target ring, $p = 0$. Hence, this ring is a
power series ring $R[[y]]$ on a generator $y$ that is represented by $t_px$
in the spectral sequence
$$E^2 = R[t_p,x] \Rightarrow
\pi_*(\THH(R,\mathbb{Z}_p)^{h\mathbb{T}_p}).$$
Hence, for $m \geq 0$, we have
$$\pi_{2m}(\THH(R,\mathbb{Z}_p)_{h\mathbb{T}_p}) = R[[y]]/y^{m+1} \cdot v_p^{-(m+1)}$$
and the remaining homotopy groups are zero.
\end{remark}

We also consider the case of the semi-direct product
$G = \operatorname{Aut}(C_p) \ltimes C_p.$
The conjugacy classes of elements in $G$ are represented by the
elements $(\id,0)$, $(\id,1)$, and $(\alpha,0)$ with
$\alpha \in \operatorname{Aut}(C_p) \smallsetminus \{\id\}$, the
centralizers of which are $G$, $C_p$, and $\text{Aut}(C_p)$,
respectively. Hence, for every $\mathbb{E}_1$-algebra in spectra $R$, 
there is a cofiber sequence of spectra with $\mathbb{T}$-action
$$\begin{aligned}
{} & \xymatrix{
{ \THH(R) \otimes BG_+^{\operatorname{triv}} } \ar[r] &
{ \THH(R[G]) } \cr
} \cr
{} & \xymatrix{
{ } \ar[r] &
{ \THH(R) \otimes BC_{p+}^{\operatorname{res}} \oplus
\THH(R) \otimes B\operatorname{Aut}(C_p)_+ \otimes \operatorname{Aut}(C_p), } \cr
} \cr
\end{aligned}$$
where the last tensor factor $\operatorname{Aut}(C_p)$ is viewed as a
pointed space with $\id$ as basepoint. Moreover, as cyclotomic
spectra, the right-hand summand splits off. Therefore, after
$p$-completion, we arrive at the following 
statement.  

\begin{theorem}\label{thm:tcsemidirectproduct}Let
$G = \operatorname{Aut}(C_p) \ltimes C_p$, and let $R$ be a connective
$\mathbb{E}_1$-ring. There is a canonical fiber sequence of spectra
$$\begin{aligned}
{} &\xymatrix{
{ \TC(R,\mathbb{Z}_p) \otimes BG_+ } \ar[r] &
{ \TC(R[G], \mathbb{Z}_p) } \cr
} \cr
{} & \xymatrix{
{ \phantom{ \TC(R,\mathbb{Z}_p) \otimes BG } } \ar[r] &
{ \THH(R,\mathbb{Z}_p)_{h\mathbb{T}_p}[1] \oplus \TC(R,\mathbb{Z}_p)
  \otimes \operatorname{Aut}(C_p), } \cr
} \cr
\end{aligned}$$
and moreover, the summand $\TC(R,\mathbb{Z}_p) \otimes
\operatorname{Aut}(C_p)$ splits off $\TC(R[G],\mathbb{Z}_p)$.
\end{theorem}
\index{group ring!topological cyclic homology of|)}
\index{topological cyclic homology!of group ring|)}

\vspace{5mm}

{\small
\noindent\textsc{Nagoya University, Japan and University of
  Copenhagen, Denmark} \hfill\space\linebreak
\textit{E-mail address}: \texttt{larsh@math.nagoya-u.ac.jp}

\vspace{2mm}

\noindent\textsc{Universit\"{a}t M\"{u}nster, Germany} \hfill\space\linebreak
\textit{E-mail address}: \texttt{nikolaus@uni-muenster.de}
}

\begin{theindex}

  \item $\infty$-category of   cyclotomic spectra, 7
  \item $\mathcal{O}_C$, 29
    \subitem $K$-theory of, 32
    \subitem topological cyclic homology of, 30
  \item $\mu$, 30
  \item $p$-adic homotopy groups, 3
  \item $p$-complete pro-infinitesimal thickening, 3, 26

  \indexspace

  \item Adams spectral sequence, 13--16
    \subitem $k$-based, 16
  \item algebraic $K$-theory
    \subitem of stable $\infty$-category, 10
  \item assembly map, 34

  \indexspace

  \item B\"{o}kstedt periodicity, 1, 3, 16--21
    \subitem for perfectoid ring, 23
  \item B\"{o}kstedt spectral sequence, 16
  \item Bhatt--Morrow--Scholze filtration, 5
  \item Bott periodicity, 29--33

  \indexspace

  \item Clausen--Mathew--Morrow theorem, 2, 33
  \item cogroupoid, 14
    \subitem associated with $\mathbb{E}_{\infty}$-ring, 14, 15
  \item cogroupoid module, 14
    \subitem associated with spectrum, 15
  \item Connes' operator, 11
    \subitem and power operations, 11
  \item cyclotomic Frobenius, 6
  \item cyclotomic spectrum, 7
  \item cyclotomic trace map, 2, 9
    \subitem on connective covers, 10

  \indexspace

  \item divided power algebra, 1
  \item Dundas--Goodwillie--McCarthy theorem, 2

  \indexspace

  \item extension of scalars, 13

  \indexspace

  \item Fargues--Fontaine curve, 23
  \item Fontaine's ring of $p$-adic periods, 22, 26

  \indexspace

  \item group ring
    \subitem topological cyclic homology of, 33--38

  \indexspace

  \item Hochschild homology, 12
  \item Hochschild--Kostant--Rosenberg filtration, 12
  \item horizontal
    \subitem with respect to stratification, 16

  \indexspace

  \item Land--Tamme theorem, 2

  \indexspace

  \item negative topological cyclic homology, 2, 8
    \subitem of perfectoid ring, 27
  \item noncommutative motives, 9
  \item Nygaard filtration, 37

  \indexspace

  \item perfectoid ring, 3, 21--29
    \subitem definition of, 23
    \subitem properties of, 23
  \item periodic topological cyclic homology, 2, 8
    \subitem of perfectoid ring, 27
  \item power operations
    \subitem Araki--Kudo, 11
    \subitem Dyer--Lashof, 11
  \item prism, 4, 23
  \item prismatic cohomology, 5
  \item pseudocoherent module, 24

  \indexspace

  \item Rees construction
    \subitem $p$-adic, 37
  \item restriction of scalars, 13

  \indexspace

  \item semiperfectoid ring, 3
  \item stratification
    \subitem relative to cogroupoid, 14
  \item symmetric monoidal product
    \subitem of modules over cogroupoid, 15
  \item syntomic cohomology, 5

  \indexspace

  \item Tate diagonal, 6
  \item Tate orbit lemma, 9, 35
  \item tilt, 22
  \item topological cyclic homology, 2, 8--10
    \subitem negative, 2, 8, 27
    \subitem of $\mathcal{O}_C$, 30
    \subitem of group ring, 33--38
    \subitem periodic, 2, 8, 27
  \item topological Hochschild homology, 6--8
    \subitem of $\mathbb{E}_{\infty}$-algebra in spectra, 6
    \subitem of stable $\infty$-categories, 7
    \subitem trace property of, 10

  \indexspace

  \item untilt, 23

\end{theindex}

\end{document}